\documentclass[11pt,a4paper]{amsart}
\usepackage{mathptmx} 
\usepackage[scaled=0.90]{helvet} 
\usepackage{courier}
\usepackage{amssymb}
\usepackage{graphicx}
\usepackage{color}
\usepackage{tikz-cd}
\normalfont
\usepackage[T1]{fontenc}

\usepackage{amssymb}

\renewcommand{\epsilon}{\varepsilon}
\renewcommand{\setminus}{\smallsetminus}
\renewcommand{\emptyset}{\varnothing}

\newtheorem{theorem}{Theorem}[section]
\newtheorem{proposition}[theorem]{Proposition}
\newtheorem{corollary}[theorem]{Corollary}
\newtheorem{lemma}[theorem]{Lemma}

\newtheorem{question}[theorem]{Question}

\newtheorem*{theorema}{Theorem A}
\newtheorem*{theoremb}{Theorem B}
\newtheorem*{theoremc}{Theorem C}
\newtheorem*{theoremd}{Theorem D}

\theoremstyle{definition}

\newtheorem{definition}[theorem]{Definition}

\theoremstyle{remark}
\newtheorem{remark}[theorem]{Remark}

\newcommand{\Ann}{\operatorname{Ann}}
\newcommand{\ad}{\mathrm{ad}}

\def\Vightarrow#1{\smash{\mathop{\longrightarrow}\limits^{#1}}}

\newcommand{\cd}{\operatorname{cd}}
\newcommand{\normal}{\lhd}

\newcommand{\gr}{\mathrm{gr}}

\newcommand{\Q}{\mathbb Q}
\newcommand{\Z}{\mathbb Z}
\newcommand{\N}{\mathbb N}

\newcommand{\FP}{\operatorname{FP}}

\newcommand{\Ho}{\operatorname{H}}
\newcommand{\cohom}[3]{H^{{\raise1pt\hbox{$\scriptstyle#1$}}}(#2\>\!,#3)}
\newcommand{\tatecohom}[3]%
  {\widehat H^{{\raise1pt\hbox{$\scriptstyle#1$}}}(#2\>\!,#3)}

\newcommand{\Cohom}[3]%
  {H^{{\raise1pt\hbox{$\scriptstyle#1$}}}\big(#2\>\!,#3\big)}
\newcommand{\Tatecohom}[3]%
  {\widehat H^{{\raise1pt\hbox{$\scriptstyle#1$}}}\big(#2\>\!,#3\big)}

\newcommand{\homol}[3]{{\mathrm H}_{{\lower1pt\hbox{$\scriptstyle#1$}}}(#2\>\!,#3)}
\newcommand{\homolog}[2]{{\mathrm H}_{{\lower1pt\hbox{$\scriptstyle#1$}}}(#2)}

\renewcommand{\ker}{\operatorname{Ker}}
\newcommand{\im}{\operatorname{Im}}

\newcommand{\Hom}{\operatorname{Hom}}

\newcommand{\Tor}{\operatorname{Tor}}

\title[]{Bass-Serre theory for Lie algebras: a homological approach }
\author{D.~ H. ~Kochloukova}

\address{Dessislava H.~Kochloukova, Department of Mathematics, University of Campinas, 
13083 - 859 Campinas, SP, Brazil}\email{desi@unicamp.br}

\author{C.~Mart\'inez-P\'erez}

\address{Conchita Mart\'inez-P\'erez, Departamento de Matem\'aticas, Universidad de Zaragoza,
50009 Zaragoza, Spain} \email{conmar@unizar.es}

\keywords{}
\subjclass[2000]{
17B55, 20J05}

\thanks{}

\begin{document}

\thispagestyle{empty}

\begin{abstract} 
We develop a version of the Bass-Serre theory for Lie algebras (over a field  $k$) via a homological approach. We define the notion of fundamental Lie algebra of a graph of Lie algebras and show that this construction yields Mayer-Vietoris sequences. We extend some well known results in group theory  to $\N$-graded  Lie  algebras: for example, we show 
 that one relator $\N$-graded Lie algebras are iterated HNN extensions with free bases which can be used for cohomology computations and apply the Mayer-Vietoris sequence to give some results about coherence of Lie algebras. \end{abstract}

\maketitle

\section{Introduction}
The  Bass-Serre Theory is one of the corner-stones of modern group theory. It has definitely contributed to the flourishing of geometric group theory, lead to many applications and it has also proven to be useful to find shorter and more elegant proofs of known results. From the point of view of group cohomology, in the core of the Bass-Serre Theory lies the fact that associated to a group action on a tree one can build a short exact sequence that yields Mayer-Vitoris exact sequences. In this paper, see Section \ref{B-S-section},  we develop a  Bass-Serre theory of Lie algebras which, if not otherwise stated  are defined over a field $k$ of arbitrary characteristic.  In the case of Lie algebras, the two basic constructions which are the building blocks of the Bass-Serre theory,  i.e. free products with amalgamation and HNN-extensions, have been already considered in the literature ( \cite[Chapter 4]{BokutKukin},  \cite{Feldman},  \cite{L-S}, \cite{Wasserman}) and we use them to define the fundamental Lie algebra of a graph of Lie algebras. 
Lie algebras do not act on trees but we show that associated to these constructions one also has short exact sequences:

\begin{theorema} Let $\Delta$ be a graph of Lie algebras with fundamental Lie algebra $L = L_{\Delta}$. Then there is a short exact sequence of right $U(L)$-modules
$$0 \to \oplus_{  e   \in  E(\Gamma)} \  k \otimes_{U(L_e)} U(L) \to \oplus_{v \in V(\Gamma)} k \otimes_{U(L_v)} U(L) \to k \to 0,$$
\end{theorema}

Here, $U(L)$ is the universal enveloping algebra of the Lie algebra $L$, see Subsection \ref{definitionfundamental} for the rest of notation.
As in the group case, this leads to Mayer-Vietoris exact sequences  (Corollary \ref{exact2}) and allows us to extend to Lie algebras some classical results for groups. The main feature of our Lie algebra version of the Bass-Serre theory is the use of homological language. 
However, there is no complete analogue for Lie algebras of the Bass-Serre theory. Indeed in \cite{Shirshov2} Shirshov showed the existence of Lie algebras such that their free Lie product has a Lie subalgebra that is not free, is not isomorphic to any subalgebra of any of the factors and cannot be decomposed as the free Lie product of any of its subalgebras.   Note that Shirshov does not use homological language. Moreover,  as Mikhalev,  Umirbaev and Zolotykh showed  in \cite{M-U-Z}  there exist non-free Lie algebras (over a field of characteristic $p > 2$)  of cohomological dimension 1. 


One of the classical applications of the Bass-Serre theory for groups is the study of one-relator groups. One-relator Lie algebras are known to share many of the (co)homological properties of one-relator groups, for example it is known that they are of type $\FP_\infty$ and have cohomological dimension at most two  \cite[Theorem 3.10.6, Corollary 3.10.7]{BokutKukin}. In the case of groups, both results can be shown by embedding one relator groups in iterated HNN extensions.  We denote $\mathbb{N} = \{ 1,2, \ldots \}$. An $\mathbb{N}$-graded Lie algebra  is a Lie algebra  $L$ with a decomposition as vector space $L = \oplus_{i \in \mathbb{N}}  L_i, \hbox{ where }[L_i, L_j] \subseteq L_{i+j}.$  In Section \ref{iterated} we show that in the $\N$-graded case, one relator Lie algebras  are iterated HNN-extensions.

\begin{theoremb} 
Let 
$L=\langle X\mid r\rangle$
be a one relator $\mathbb{N}$-graded  Lie algebra  with $X$ finite, where  the $\mathbb{N}$-grading is induced by  some weight function $\omega:X\to\Z^+$, i.e.  $r\in F(X)$ is homogeneous  with respect to $\omega$. Then $L$ is an iterated HNN-extension  of Lie algebras
$$L=(\ldots(A*_{h_n})*_{h_{n-1}})\ldots)*_{h_1}$$
such that $A$ and all the associated  Lie subalgebras in each of the HNN-extensions are free.
 \end{theoremb}
This result combined with Theorem A can be used for explicit computations in cohomology.

Coherence is another group theoretical property deeply related to Bass-Serre theory. For example, fundamental groups of graphs of groups where the edge groups have the property that all its subgroups are finitely generated and all its vertex groups are coherent are known to be coherent \cite[Lemma 4.8, page 41]{Wilton} and an analogous result holds for the corresponding group rings  \cite{A}, \cite{Lam}. Droms  \cite{D} has characterised the right angled Artin groups $A_\Gamma$ which are coherent. Here we consider similar results for Lie algebras and their universal enveloping algebras, using as a main tool the Mayer-Vietoris sequences constructed in Section \ref{B-S-section}.

\begin{theoremc}
Assume that the Lie algebra $L$ is the fundamental  Lie algebra of a graph of Lie algebras such that for the vertex Lie algebras $L_v$ the universal enveloping algebra $U(L_v)$ is coherent and for each edge algebra $L_e$, $U(L_e)$ is Noetherian. Then $U(L)$ is coherent.
\end{theoremc}

Let $\Gamma$ be a finite simple graph with vertex set $V(\Gamma)$ and edge set $E(\Gamma)\subset V(\Gamma)\times V(\Gamma)$, we assume that there are no loops or double edges. Associated to $\Gamma$ there is  a group called the {\it right angled Artin group} $G_\Gamma$ given by the presentation  in the category of groups
$$G_\Gamma=\langle V(\Gamma)\mid [u,v]=1\text{ whenever }\{u,v\}\in E(\Gamma)\rangle$$
and also a $k$-Lie algebra, we  call the right angled Artin Lie algebra $L_\Gamma$, given by the presentation in the category of Lie algebras
$$L_\Gamma=\langle V(\Gamma)\mid [u,v]=0\text{ whenever }\{u,v\}\in E(\Gamma)\rangle.$$

 The algebra $L_{\Gamma}$ is naturally graded  and generated in degree 1 by the elements of $V(\Gamma)$.

\begin{theoremd}  Let $L_\Gamma$ be a right angled Artin Lie algebra.  Then the following conditions are equivalent :

1) $ U(L_\Gamma)$ is coherent; 

2) the graph $\Gamma$ is chordal, i.e., it has no $n$-cycle embedded as a full subgraph for $n\geq 4$;

3) every finitely generated $\mathbb{N}$-graded subalgebra of $L_{\Gamma}$ is $FP_{\infty}$. 
\end{theoremd}

The main examples of Lie algebras considered  in the first sections of paper (one-relator and right angled Artin Lie algebras)   are  defined using presentations in terms of generators and relations  in the category of Lie algebras  and are precisely the Lie algebras associated to the descending central series of the analogous groups  after tensoring with $\otimes_{\mathbb{Z}} k$. In all these cases  one sees a very close relation between the (co)homological finiteness properties of the Lie algebra and those of the group. 

 Theorem D shows that with respect to coherence $\mathbb{N}$-graded Lie algebras resemble discrete groups. In section \ref{example} we will  see an example that shows that right angled Artin Lie algebras do not always have the same properties as right angled Artin groups and it is linked with the fact  that for $\mathbb{N}$-graded Lie algebras Kurosh type result does not hold. The fact that a Kurosh type result does not hold for general Lie algebras was observed earlier by Shirshov \cite{Shirshov2}. As well we will show an example of $\mathbb{N}$-graded Lie algebra that has infinitely many ends but does not split as a free product of $\mathbb{N}$-graded Lie algebras.  It was shown by Feldman in \cite{Feldman2} that the Shirshov example is a Lie algebra with infinitely many ends that is not a free product but Shirshov's example is not $\mathbb{N}$-graded.

Unless otherwise stated, all along the paper by module we mean {\it rigth} module. 

\medskip
{\bf Acknowledgements} During the preparation of this work the first named author was partially supported by CNPq grant 301779/2017-1  and by FAPESP grant 2018/23690-6. The second named author was partially supported by PGC2018-101179-B-I00 and by Grupo \'Algebra y Geometr\'ia, Gobierno de Arag\'on and Feder 2014-2020 \lq\lq Construyendo Europa desde Arag\'on".

\section{Bass-Serre theory for Lie algebras} \label{B-S-section}

In this section we develop a version of  Bass-Serre theory for Lie algebras. In the context of Lie algebras we can not talk about actions but we do have all the (co)homological consequences that in the case of groups one can obtain by using actions on trees. More precisely, from the action  of a group  on a tree one can derive the existence of certain exact sequences that can be used to prove (co)homological results. Here, we will also derive similar sequences using different methods.

 We begin by considering the Lie algebra version of the two constructions which are the basis for Bass-Serre theory: amalgamated free products and HNN extensions.

Throughout the paper all Lie algebras are over a field $k$, for a Lie algebra $L$  we denote by $U(L)$ the universal enveloping algebra of $L$ and $$\epsilon : U(L) \to k$$ denotes the augmentation map that sends $L$ to 0 and is the identity on $k$.  We say that $$L\hbox{  has a presentation (in terms of generators and relations) }\langle X | R \rangle \hbox{ if }L \simeq F/ N,$$ where $F$ is the free Lie algebra with a free basis $X$, $R$ is a subset of $F$ and $N$ is the ideal in $F$ generated by $R$. Sometimes instead of $R$ we write a set of relations $\{ \alpha_j = \beta_j \mid j\in J \}$ that should be interpreted as  $R = \{ \alpha_j - \beta_j \mid j\in J \}$. If $M_i = \langle X_i | R_i \rangle$ is a Lie algebra for $i \in I$  then   $\langle \cup_i X_i |  (\cup_{i \in I} R_i)  \cup \{ \alpha_j - \beta_j \}_{j \in J} \rangle$ is denoted by $\langle \cup_{i \in I} M_i | \alpha_j = \beta_j \hbox{ for } j \in J \rangle$ or by  $\langle M_1, \ldots, M_n | \alpha_j = \beta_j \hbox{ for } j \in J \rangle$ if $I = \{ 1, \ldots, n \} $.

\subsection{Amalgamated Lie algebra products and the associated short exact sequence} \label{prod}

 Let $L_0,L_1,L_2$ be Lie algebras and assume that we have monomorphisms $\sigma:L_0\to L_1$, $\tau:L_0\to L_2$. The  {\it free amalgamated product } of the Lie algebras $L_1$ and $L_2$   with amalgam $L_0$ is the Lie algebra given by the presentation in terms of generators and relations
$$L = L_1 *_{L_0} L_2=\langle L_1,L_2\mid\sigma(a)=\tau(a),a\in L_0\rangle.$$

\begin{proposition} \label{amalgam*} Let  $L = L_1 *_{L_0} L_2$ be  a free amalgamated product of Lie algebras. Then there is an exact complex of right $U(L)$-modules $$ 0 \to k \otimes_{U(L_0)} U(L)  \Vightarrow{\alpha}
	(k \otimes_{U(L_1)} U(L))  \oplus  (k \otimes_{U(L_2)} U(L))  \Vightarrow{ \beta} k \to 0,  
	$$
	where $\alpha( 1 \otimes \lambda) = (1 \otimes \lambda, - 1 \otimes \lambda)$ and $\beta(1 \otimes \lambda_1, 1 \otimes \lambda_2) = \epsilon(\lambda_1) + \epsilon(\lambda_2)$.		\end{proposition}
\begin{proof} Note that by \cite[Def.~ 4.2.1, Thm.~ 4.4.2]{BokutKukin}, the canonical maps $L_1,L_2\to L_1*_{L_0}L_2$ are injective.
By \cite[Prop.~3]{Feldman}  there is a short exact sequence of left $U(L)$-modules
$$
0 \to U(L) L_0  \Vightarrow{\partial_1}  U(L) L_1 \oplus U(L) L_2 \Vightarrow{\partial_0} U(L) L \to 0
$$
given by $$\partial_1(\lambda_0) = (\lambda_0, - \lambda_0) \hbox{ and } \partial_0 (\lambda_1, \lambda_2) = \lambda_1 + \lambda_2 \hbox{ for } \lambda_i \in U(L) L_i, ~ 0 \leq i \leq 2,  $$
where we consider $L_0$ as a Lie subalgebra of both $L_1$ and $L_2$ and consider  $L_1$ and $L_2$ as Lie subalgebras of $L$. 
Similarly there is a  short exact sequence of right $U(L)$-modules
$$
{\mathcal C}_1 : 0 \to  L_0 U(L) \to  L_1  U(L) \oplus  L_2  U(L) \to  L U(L) \to 0.
$$
The short exact sequence ${\mathcal C}_1$ can be embedded in a short exact sequence of right $U(L)$-modules
$$
{\mathcal C} : 0 \to U(L) \to U(L) \oplus U(L) \to U(L) \to 0$$
sending $\lambda \in U(L)$ to $(\lambda, - \lambda)$, and $(\lambda_1, \lambda_2) \in U(L) \oplus U(L)$ to $\lambda_1 + \lambda_2$. Thus we have a short exact sequence of complexes of right $U(L)$-modules
$$
0 \to {\mathcal C}_1 \to {\mathcal C} \to {\mathcal C}/ {\mathcal C}_1 \to 0
$$
By the $3 \times 3$ lemma \cite[Exer.~2.32]{Rotman} since  two of the complexes above are exact ( namely ${\mathcal C}_1$ and  ${\mathcal C}$) we deduce that the third complex 
$${\mathcal C}/ {\mathcal C}_1 : 0 \to k \otimes_{U(L_0)} U(L) \to 
 (k \otimes_{U(L_1)} U(L))  \oplus  (k \otimes_{U(L_2)} U(L)) \to  k \to 0  
$$ is exact. This is exactly the complex from the statement.
\end{proof}

\subsection{HNN-extensions of Lie algebras  and the associated short exact sequence}  \label{sec:HNN}

HNN extensions of Lie algebras were considered by Lichtman-Shirvani in \cite{L-S} and independently by Wasserman in \cite{Wasserman}.

 Let $L$ be a Lie algebra and take $A \leq L$ a Lie subalgebra.  
A {\it derivation} $d:A \to L$ is a $k$-linear map such that
$$d([a,b])=[a,d(b)]+[d(a),b]$$
for any $a,b\in A$. Given a derivation $d : A \to L$ we set
\begin{equation}\label{eq:HNN} W =\langle L ,t | [t,a]=d(a)\text{ for }a\in  A\rangle.\end{equation}
We call $W$ a Lie algebra {\it  HNN-extension} with  {\it base} Lie subalgebra $L$,   {\it associated } Lie subalgebra  $A$ and {\it stable letter} $t$.
By \cite{Wasserman} the canonical map $ L \to W$ is injective.

\begin{proposition} \label{lem:HNN} Let $W$ be an HNN extension Lie algebra $W = \langle L, t \mid [t,a] = d(a) \hbox{ for } a \in A \rangle$. Then there is a short exact sequence of right $U(W)$-modules
	$$ 0 \to k\otimes_{U(A)}U(W)   \Vightarrow{\alpha} k\otimes_{U(L)}U(W) \Vightarrow{\beta} k \to 0$$
	given by $\alpha(1\otimes \lambda) = 1\otimes t\lambda$ and $\beta (1\otimes \lambda) = \epsilon(\lambda)$, where $\epsilon : U(W)\to k$ is the augmentation map. 
\end{proposition}
\begin{proof} The fact that $\beta$ is surjective is obvious. It is also obvious that $\im\alpha\subseteq\ker\beta$. Consider the augmentation map $\epsilon : U(W)\to k$. Its kernel is the augmentation ideal $WU(W) =  U(W) W$, so we have a short exact sequence 
	$$0\to W U(W) \to U(W) \Vightarrow{\epsilon} k\to 0.$$
	Applying the right exact functor $k\otimes_{U(L)} -$ to this sequence yields an exact sequence
	$$k\otimes_{U(L)}(WU(W))\to k\otimes_{U(L)}U(W)\buildrel\beta\over\to k\to 0$$
	and we deduce that $\ker\beta$ is the image of $k\otimes_{U(L)}(WU(W))$. Take any element $0\neq b \in \ker(\beta)$. It must be
	in the image of $k\otimes_{U(L)}(WU(W))$, so it is expressible
	as a sum of monomials of the form $1\otimes \lambda$, where $\lambda$ is a word written associatively in the generators of $W$. As the generators of $W$ are the generators of $L$ together with the element $t$, we see that  for each
	$1\otimes \lambda=0$ in $k\otimes_{U(L)}U(W)$ we may assume $\lambda=t\lambda_1$ for some other monomial $\lambda_1$ in the generators of $W$. This means that the element $b$ is a sum of monomials of the form $1\otimes t\lambda_1$ so it lies in $\im(\alpha)$.
	
	We have to show that $\alpha$ is a monomorphism. Note that this will follow if we prove this: 
	
	\medskip
	{\bf Claim.} {\it  For any $\lambda\in U(W)$ such that $t\lambda\in LU(W)$, we have $\lambda\in AU(W)$.}
	
	\medskip
	Let $\Gamma_1$ be a (possibly infinite) generating set of the Lie algebra $L$ containing a generating set $\Gamma_2$ of  the Lie algebra $A$. Set $\Gamma = \Gamma_1\cup\{t\}$. Let $F$ be a free Lie algebra (over the field $k$) with  free basis $\Gamma$. Then $F$ subjects onto $W$ via a map that is the identity on $\Gamma$. Let $R$ be the kernel of this surjection, hence $R$ is an ideal of $F$ such that $W=F/R$. Consider the standard resolution of the trivial right $U(F)$-module $k$ 	
	\begin{equation} \label{ses*} 0\to\oplus_{\Gamma}U(F)\to U(F)\to k\to 0 \end{equation} 
	and apply the functor $- \otimes_{U(R)}k$. Observe that $U(F)$ is free as $U(R)$-module (via multiplication) and that
	$U(W)=U(F)\otimes_{U(R)}k$. Therefore by the long exact sequence in homology associated to (\ref{ses*}) we get an exact sequence
	$$0=\Tor^{U(R)}_1(U(F),k)\to\Tor^{U(R)}_1(k,k)\to\oplus_{\Gamma}U(W)\to U(W)\to k\to 0$$
	and as $\Tor^{U(R)}_1(k,k) \simeq R/[R,R]=R_{\text{ab}}$ the exact sequence above is
	\begin{equation} \label{exact*} 0\to R_{\text{ab}}\buildrel{\delta_1}\over\to\oplus_{\Gamma}U(W)\buildrel{\delta_0}\over\to U(W)\to k\to 0.\end{equation}
	We analyze the map $\delta_0$. We denote $$\oplus_{\Gamma}U(W)=\oplus_{\gamma\in\Gamma_1}e_{\gamma}U(W)\oplus e_tU(W),$$ then $\delta_0(e_\gamma)=\gamma$ for $\gamma \in \Gamma_1$ and $\delta_0(e_t)=t$.
	Now, let $\lambda\in U(W)$ be such that $t\lambda\in LU(W)$. This means that
	$$t\lambda=\sum_{\gamma \in \Gamma_1} \gamma\lambda_{\gamma}$$
	with $\lambda_{\gamma} \in U(W)$. Then
	$$\delta_0(e_t\lambda-\sum_{\gamma\in\Gamma_1}e_\gamma\lambda_\gamma)=t\lambda-\sum_{\gamma\in\Gamma_1} \gamma\lambda_\gamma=0,$$
	so by the exactness of (\ref{exact*}) 
	$$e_t\lambda-\sum_{\gamma\in\Gamma_1}e_\gamma\lambda_\gamma=\delta_1(\omega)$$
	for some $\omega\in R_{\text{ab}}.$ 
	
	Note that any $r \in R$ can be writen as a linear combination of Lie monomials on $\Gamma$ and then in $U(W)$ we can write $r$ as $\sum_{i_1, \ldots, i_k} k_{i_1, \ldots, i_k} \gamma_{i_1} \ldots \gamma_{i_k}$.  Therefore $$\delta_1(r + [R,R]) = \sum_{i_1, \ldots, i_k} k_{i_1, \ldots, i_k} e_{\gamma_{i_1}} \gamma_{i_{2}} \ldots \gamma_{i_k}$$
	Recall that the ideal $R$ is generated by the relators of $L$  and by relators of the form $[t,a]-d(a)$ for $a\in \Gamma_2$. By the above description $\delta_1$ sends the images of the relators of $L$ to elements of $\sum_{\gamma \in \Gamma_1} e_{\gamma} U(W)$.
	Since  $[t,a]-d(a)=ta-at-d(a)$ we deduce that for $r_0 = ([t,a]-d(a)) \circ v$ where $\circ$ denotes the adjoint action, $v \in U(W)$ and  $a \in \Gamma_2$ we have
	$$\delta_1(r_0 +[R,R]) \in e_t av - e_a tv   + \sum_{\gamma \in \Gamma_1} e_{\gamma} U(W).$$ Hence
	$$
	e_t\lambda-\sum_{\gamma\in\Gamma_1}e_\gamma\lambda_\gamma= \delta_1(\omega) \in \sum_{a \in \Gamma_2} (e_a t - e_t a) U(W) + \sum_{\gamma \in \Gamma_1} e_{\gamma} U(W) \subseteq $$ $$e_t A U(W) + \sum_{\gamma \in \Gamma_1} e_{\gamma} U(W) = e_t A U(W) \oplus ( \oplus_{\gamma \in \Gamma_1} e_{\gamma} U(W))
	$$
So $\lambda\in AU(W)$ as we wanted to prove.
	
\end{proof}

\subsection{Fundamental Lie algebra of a graph of Lie algebras and the associated short exact sequence}\label{definitionfundamental}

 A {\it  finite graph of Lie algebras} $\Delta$ is a set
$$\Delta = \Delta (V(\Delta),E(\Delta), \Gamma, T,\{L_v\}_{v\in V}, \{L_e\}_{e\in E(\Gamma)},\{ \sigma_e \}_{e \in E(\Gamma)},\{\tau_e\}_{e\in E(T)}, \{ d_e\}_{e \in E(\Gamma)\setminus E(T) })$$
with the following data. We have finite sets $V(\Delta)$ (vertices) and $E(\Delta)$ (edges) so that elements of $E(\Delta)$ are ordered pairs $e=(v,w)$ in $V(\Delta)\times V(\Delta)$. For an edge $e\in E(\Delta)$ we denote by $\sigma(e)$ (resp. $\tau(e)$) the beginning (resp. the end of $e$). As usual $\overline{e}$ is the  inverse edge so that $\tau(\overline{e}) = \sigma(e)$ and $\sigma(\overline{e}) = \tau(e)$ and whenever $e\in E(\Delta)$, then also $\bar{e}\in E(\Delta)$. A loop is an edge such that $\bar e=e$.
 $\Gamma$ is an underlying oriented finite  graph. It is a graph with vertex set $V(\Gamma)=V(\Delta)$ and as edge set $E(\Gamma)$ we choose exactly one of each pair $\{e,\bar e\}$ so that all loops are in $\Gamma$.   We fix  a maximal  forest $T$ in $\Gamma$ i.e. $T$ is a subgraph of $\Gamma$ which is a disjoint union of trees such that $V(T)=V(\Gamma)$ and which is maximal under these conditions. For every vertex $v \in V(\Gamma)$ we have a Lie algebra $L_v$ and for every edge $e\in E(\Gamma)$ we have a Lie algebra $L_e$ and a
  monomorphism of Lie algebras $$\sigma_e : L_e \to L_{\sigma(e)}.$$
  Moreover, if $e \in E(T)$ we also have a monomorphism
  $$\tau_e : L_e \to L_{\tau(e)}$$ 
  and if $e \in E(\Gamma)\setminus E(T)$ we have a derivation 
  $$d_e : L_e \to L_{\tau(e)}.$$

Associated to a graph of Lie algebras as before, consider the Lie algebra given by the presentation (in terms of generators and relators)
$$
L_{T} = \langle  \cup_{v \in V(\Gamma)} L_v \mid  \sigma_e(a) = \tau_e(a) \hbox{ for } e \in E(T), a \in L_e \rangle
$$
i.e. we make a free product with amalgamation for every edge e of $T$.  We define the fundamental Lie algebra of the graph of Lie algebras $\Delta$ as the Lie algebra given by the presentation
$$
L_{\Delta}  = \langle L_T,  \{ e \}_{e \in A} \mid  [e, \sigma_e(a)] = d_e (a)
 \hbox{ for } e \in A, a \in  L_e \rangle
$$
Note that in the definition of the corresponding notion for groups the choice of the maximal forest is not important as changing the maximal tree produces an isomorphic group. In fact, there is more symmetry in the group case since instead of derivations $d_e$ we have monomorphisms. In the Lie algebra case the situation is not symmetric  since derivations are not homomorphisms of Lie algebras and in the original graph of Lie algebras $\Delta$ the maximal forest $T$  and the orientation  of the edges in $E(\Gamma)\setminus E(T)$ has to be fixed explicitly  in order to define the derivations $d_e$. The precise orientation of the edges in $T$ is not important, though.

\begin{lemma} Let  $\Delta $ be a  finite graph of Lie algebras with fundamental Lie algebra $L_{\Delta}$. Then for any vertex $v$ the map $L_v \to L_{\Delta}$ is a monomorphism.
\end{lemma}

\begin{proof} This is known if $\Gamma$ is a single edge because in that case the fundamental Lie algebra is  either a free product with amalgamation or, in the case when the edge is a loop, an HNN extension. For HNN extensions it follows  by the work of Wasserman that the base of the HNN extension embeds in the HNN extension \cite{Wasserman}.  In the case of an amalgamated product of two Lie algebras each one of the two Lie algebras embeds in the amalgamated product by \cite[Thm.~4.4.2]{BokutKukin}.

In the general case we can induct on the size of $\Gamma$ which is $|V(\Gamma)| + | E(\Gamma)|$. Fix an edge $e_0$ of  $T$ with vertices $v_1$ and $v_2$.  Consider a new graph $\widetilde{\Gamma}$ obtained from $\Gamma$ by squashing the edge $e_0$ to a point $v_0$ and let $\widetilde{T}$ be the maximal forest of $\widetilde{\Gamma}$ that is obtained from $T$ by squashing the edge $e_0$ to $v_0$. Then there is a bijection between  $E(\Gamma)\setminus E(T)$ and its image $E(\widetilde{\Gamma})\setminus E(\widetilde{T})$ and 
	 there is an obvious way to construct a new graph of Lie algebras $\widetilde{\Delta}$ with  underlying oriented graph $\widetilde{\Gamma}$  where for $v \in V(\Gamma) \setminus \{ v_1, v_2 \}$ the associated Lie algebra in  $\widetilde{\Delta}$ is the old Lie algebra $L_v$ and $L_{v_0} = L_{v_1} *_{L_{e_0}} L_{v_{2}}$.  The set of edges $\widetilde{E}$ of $\widetilde{\Delta}$ is in bijection with $E(\Delta) \setminus \{ e_0, \overline{e}_0 \}$ and the edge Lie algebras in $\widetilde{\Delta}$ are defined  to be the old ones i.e. $\{ L_e  \mid e \in E \setminus \{ e_0, \overline{e}_0 \} \}$ and the monomorphisms and the derivations from the edge Lie algebras in $\widetilde{\Delta}$ to the corresponding vertex Lie algebras in $\widetilde{\Delta}$ are  the old ones from $\Delta$ combined if necessary with the embeddings of $L_{v_1}$ and $L_{v_2}$
in $L_{v_0} = L_{v_1} *_{L_{e_0}} L_{v_2}$.  Thus $L_{\Delta} = L_{\widetilde{\Delta}}$ and, as $\widetilde{\Gamma}$ has smaller size than $\Gamma$, by induction the canonical map from a Lie algebra associated to a vertex in $\widetilde{\Delta}$ to $L_{\widetilde{\Delta}}$ is injective. In particular this is the case for the Lie algebra $L_{v_0}$ and as $L_{v_1}$, $L_{v_2}$ also embed in $L_{v_0}$ we get the result.

Finally if $E(T) = \emptyset$ we deduce that  $E ( \Delta)$ contains only loops. Thus $\Gamma = \cup_i  \Gamma_i$  is a disjoint union of graphs where  each $V(\Gamma_i)$ contains precisely one vertex, say $v_i$, and we have the corresponding decomposition $\Delta = \cup_i \Delta_i$ of graphs of Lie algebras. Then  $L_{\Delta}$ is the free product of the Lie algebras $L_{\Delta_i}$ and each $L_{\Delta_i}$ is obtained from $L_{v_i}$ by applying several times HNN-extensions with stable letters corresponding to loops from $E(\Gamma_i)$. Finally $L_{v_i}$ embeds in $L_{\Delta_i}$ and $L_{\Delta_i}$ embeds in $L_{\Delta}$.
\end{proof}

We can prove now Theorem A from the introduction.

\begin{theorem} \label{exact1} Let $\Delta$ be  a  finite  graph of Lie algebras with  fundamental Lie algebra $L = L_{\Delta}$. Then there is a short exact sequence of right $U(L)$-modules
$$0 \to \oplus_{  e  \in  E(\Gamma)} \  k \otimes_{U(L_e)} U(L) \to \oplus_{v \in V(\Gamma)} k \otimes_{U(L_v)} U(L) \to k \to 0,$$
\end{theorem}

\begin{proof} It suffices to prove the result when the underlying graph is just one edge (i.e. consider two cases: free product with amalgamation and HNN extension), then use the definition with the maximal forest  and induction on the number of vertices. The two basic cases  free product with amalgamation and HNN extension were proved in Subsections \ref{prod} and \ref{sec:HNN}. We explain the details below.
	
\noindent {\bf Case 1}. Suppose $\Gamma = T$. We induct on the size of $T$. The induction starts with $T$ a vertex and the result is obvious in this case. Assume now that $T$ is obtained from a forest $T_0$ by adding  an edge $e_0$ 
such that $T$ and $T_0$ have the same number of connected components. Let $\Delta_0$ be the graph of Lie algebras with underlying graph $\Gamma_0  = T_0$ and vertex, edge Lie algebras as in $\Delta$ and the same monomorphisms as in $\Delta$. Then for $L_0 = L_{\Delta_0}$ and $L = L_{\Delta}$ we have $L = L_0 *_{L_{e_0}} L_{v_0}$, where $v_0$ is the vertex of $e_0$ that is not a vertex of $T_0$. This gives a short exact sequence of $U(L)$-modules
	\begin{equation} \label{ses-oberwolfach1} 0 \to k \otimes_{U(L_{e_0})} U(L) \to k \otimes_{U(L_{v_0})} U(L) \oplus k \otimes_{U(L_{0})} U(L) \to k \to 0\end{equation} 
	On the other hand by induction we know the  result holds true for $\Delta_0$ i.e. there is a short exact sequence of $U(L_0)$-modules
	$$
	0 \to \oplus_{  e \in E(\Gamma_0)} \  k \otimes_{U(L_e)} U(L_0) \to \oplus_{v \in V(\Gamma_0)} k \otimes_{U(L_v)} U(L_0) \to k \to 0$$ and by applying the exact functor $- \otimes_{U(L_0)} U(L)$ we get the short exact sequence
	\begin{equation} \label{ses-oberwolfach2}
		0 \to \oplus_{  e \in  E(\Gamma_0)} \  k \otimes_{U(L_e)} U(L) \to \oplus_{v \in V(\Gamma_0)} k \otimes_{U(L_v)} U(L) \to k \otimes_{U(L_0)} U(L) \to 0 \end{equation} 
		Gluing (\ref{ses-oberwolfach1}) and (\ref{ses-oberwolfach2}) along the module $k \otimes_{U(L_0)} U(L)$ we obtain a short exact sequence of $U(L)$-modules
		$$
			0 \to k \otimes_{U(L_{e_0})} U(L)  \oplus (\oplus_{  e \in E(\Gamma_0)} \  k \otimes_{U(L_e)} U(L)) \to$$ $$ k \otimes_{U(L_{v_0})} U(L)  \oplus ( \oplus_{v \in V(\Gamma_0)} k \otimes_{U(L_v)} U(L)) \to k  \to 0
		$$
		that is precisely  the result we want to prove. 
		
\noindent {\bf Case 2}. We induct on the size of $E(\Gamma) \setminus E(T)$, the case $\Gamma = T$ being the starting point of the induction. Assume for the inductive step that $\Gamma_0$ is a subgraph of $\Gamma$ that contains $T$ such that $\Gamma$ is obtained from $\Gamma_0$ by adding  an edge $e_0$. 
We can define as before $\Delta_0$ to be the graph of Lie algebras with underlying graph $\Gamma_0$ and all structural data: edge and vertex Lie algebras,  the structure monomorphisms and derivations as in $\Delta$. Then for $ L = L_{\Delta}$ and $L_0 = L_{\Delta_0}$ we have that $L$ is an HNN extension of $L_0$ with stable letter $e_0$ i.e.
		$$
		L = \langle L_0, e_0 \mid [e_0, \sigma_{e_0}(b)] = d_{e_0} (b) \hbox{ for } b \in L_{e_0} \rangle.
		$$
		Thus we have a short exact sequence of $U(L)$-modules				\begin{equation} \label{ses-oberwolfach3}
		0 \to k \otimes_{U(L_{e_0})} U(L) \to k\otimes_{U(L_0)} U(L) \to k \to 0.
		\end{equation} 
		By induction the result holds for $\Delta_0$, hence there is a short exact sequence of $U(L_0)$-modules that after applying the exact functor $- \otimes_{U(L_0)} U(L)$ yields the short exact sequence
		\begin{equation} \label{ses-oberwolfach4}
			0 \to \oplus_{  e \in E(\Gamma_0)} \  k \otimes_{U(L_e)} U(L) \to \oplus_{v \in V(\Gamma_0)} k \otimes_{U(L_v)} U(L) \to k \otimes_{U(L_0)} U(L) \to 0 \end{equation} 
			Gluing the exact sequences (\ref{ses-oberwolfach3}) and (\ref{ses-oberwolfach4}) along $k \otimes_{U(L_0)} U(L)$  gives the exact sequence 
			$$
			0 \to k \otimes_{U(L_{e_0})} U(L) \oplus (\oplus_{  e \in  E(\Gamma_0)} \  k \otimes_{U(L_e)} U(L)) \to \oplus_{v \in V(\Gamma_0)} k \otimes_{U(L_v)} U(L) \to k  \to 0,
			$$
			which is precisely the result we want to prove.
	\end{proof}

	If $L$ is a Lie algebra then $U(L)$ is a Hopf algebra with comultiplication  $\Delta : U(L) \to U(L) \otimes U(L)$  given by $\Delta(g)=1\otimes g+g\otimes 1$ for $g\in L$. This fact can be used to define a $U(L)$-module structure in the $k$-tensor product of two right  $U(L)$-modules $U$ and $V$ via
$$(u\otimes v)g:=ug\otimes v+u\otimes vg$$
for $g\in L$, $u\in U$ and $v\in V$. Moreover this combined with induction for a $U(T)$-module $W$ and  $T \leq L$ a Lie subalgebra yields an isomorphism of $U(L)$-modules
\begin{equation}\label{preMackey}(W\otimes_k V) \otimes_{U(T)} U(L)\cong (W \otimes_{U(T)} U(L))\otimes_k V.\end{equation}
For arbitrary elements we use the comultiplication $\Delta$ to define the isomorphism in (\ref{preMackey}), 
i.e. if for $\lambda \in U(L)$ we have
$$
\Delta(\lambda) = \sum_i \lambda_{i,1} \otimes \lambda_{i,2},
$$
where $\lambda_{i,j} \in U(L)$, the isomorphism from (\ref{preMackey}) is defined by
$$
(w \otimes v) \otimes  \lambda  \mapsto \sum_i (w \otimes  \lambda_{i,1} ) \otimes v \lambda_{i,2} .
$$
In particular, (\ref{preMackey}) implies  for $W = k$ that
\begin{equation}\label{Mackey}V \otimes_{U(T)} U(L)\cong (k \otimes_{U(T)} U(L))\otimes_k V.\end{equation}

Theorem \ref{exact1} implies the following corollaries.
\begin{corollary} \label{ses-homology} Let $L$ be a Lie algebra and  $V$ a  right $U(L)$-module.  Then there is a short exact sequence of right $U(L)$-modules
		\begin{equation}\label{secGeneral}0 \to \oplus_{  e \in  E(\Gamma)} V \otimes_{U(L_e)} U(L) \to \oplus_{v \in V(\Gamma)} V \otimes_{U(L_v)} U(L) \to V \to 0,\end{equation}
	and we get, for any  left $U(L)$-module $A$, a long exact sequence 
	$$ \to \oplus_{  e \in  E(\Gamma)}  \Tor_i^{U(L_e)}(V, A) \to  \oplus_{v \in V(\Gamma)}\Tor_i^{U(L_v)}(V, A) \to\Tor_i^L(V,A) \to \oplus_{  e \in  E(\Gamma)}  \Tor_{i-1}^{U(L_e)}(V, A) \to $$
	$$ \ldots \to\Tor_1^L(V,A) \to \oplus_{  e \in  E(\Gamma)} \Tor_0^{U(L_e)}(V, A) \to  \oplus_{v \in V(\Gamma)}\Tor_0^{L_v}(V, A) \to\Tor_0^L(V,A) \to 0,$$
	and an analogous long exact sequence of Ext functors.
	\end{corollary}

	\begin{proof}

	Applying the exact functor $-\otimes_k V$ to the short exact sequence of Theorem \ref{exact1} together with (\ref{Mackey}) yields the short exact sequence
	(\ref{secGeneral}).  This induces,  for any left $U(L)$-module $A$, a long exact sequence in $\Tor$
	$$
	\ldots \to \Tor_i^{U(L)}(\oplus_{  e \in  E(\Gamma)} V \otimes_{U(L_e)} U(L), A)  \to \Tor_i^{U(L)} (\oplus_{v \in V(\Gamma)}  V \otimes_{U(L_v)} U(L), A) \to$$ $$ \Tor_i^{U(L)}(L,A) \to  \Tor_{i-1}^{U(L)}(\oplus_{  e \in  E(\Gamma)}  V \otimes_{U(L_e)} U(L), A )  \to \ldots$$
	Note that $\Tor_i^{U(L)}( \oplus_{  e \in  E(\Gamma)}  V \otimes_{U(L_e)} U(L), A) \simeq \oplus_{  e \in  E(\Gamma)}
	\Tor_i^{U(L)}(  V \otimes_{U(L_e)} U(L), A) $ and by a version of Shapiro Lemma for Lie algebras $\Tor_i^{U(L)}(V \otimes_{U(L_e)} U(L), A) \simeq \Tor_i^{U(L_e)}(V, A)$. Similarly 
	$\Tor_i^{U(L)}( \oplus_{v \in V(\Gamma)}  V \otimes_{U(L_v)} U(L), A) \simeq \oplus_{v \in V(\Gamma)} \Tor_i^{U(L_v)}(V, A)$. 
	and we  also get  a long exact sequence 
	$$ \to \oplus_{  e \in  E(\Gamma)}  \Tor_i^{U(L_e)}(V, A) \to  \oplus_{v \in V(\Gamma)}\Tor_i^{U(L_v)}(V, A) \to\Tor_i^{U(L)} (V,A) \to$$ $$ \oplus_{  e \in  E(\Gamma)}  \Tor_{i-1}^{U(L_e)}(V, A) \to  \ldots \to\Tor_1^{U(L)} (V,A) \to \oplus_{  e \in  E(\Gamma)} \Tor_0^{U(L_e)}(V, A) \to $$ $$ \oplus_{v \in V(\Gamma)}\Tor_0^{U(L_v)}(V, A) \to\Tor_0^L(V,A) \to 0.$$
	\end{proof}
	
	\begin{remark} In the particular case when $L=L_v*_{L_e}$ is an HNN-extension and the $U(L)$-module $V$ is $U(L)$, (\ref{secGeneral}) is 
		$$0 \to U(L) \otimes_{U(L_e)} U(L) \buildrel f\over\to U(L) \otimes_{U(L_v)} U(L) \to U(L) \to 0$$
		and one can check that the map $f$ is given by $f(1\otimes 1)=1\otimes t-t\otimes 1$. This map can be described as the linearization of the map that yields a similar short exact sequence for a ring HNN extension in \cite[(1) pg. 438]{Dicks}, where $1\otimes 1\mapsto 1\otimes 1-t\otimes t^{-1}$.
	\end{remark}
	
	\begin{corollary} \label{exact2} Let $\Delta$ be a graph of Lie algebras with  fundamental Lie algebra $L = L_{\Delta}$. Then for any  right $U(L)$-module $B$ and any  left $U(L)$-module $A$ there is 
		
		a) a long exact sequence in homology
		$$ \to \oplus_{  e \in  E(\Gamma)}  \Ho_i(L_e, A) \to  \oplus_{v \in V(\Gamma)} \Ho_i(L_v, A) \to \Ho_i(L,A) \to \oplus_{  e \in  E(\Gamma)}  \Ho_{i-1}(L_e, A) \to $$
		$$ \ldots \to \Ho_1(L,A) \to \oplus_{  e \in  E(\Gamma)} \Ho_0(L_e, A) \to  \oplus_{v \in V(\Gamma)} \Ho_0(L_v, A) \to \Ho_0(L,A) \to 0$$
		
		b) a long exact sequence in cohomology
		$$0 \to \Ho^0(L,B ) \to \oplus_{v \in V(\Gamma)} \Ho^0(L_v, B) \to  \oplus_{  e \in  E(\Gamma)}  \Ho^0(L_e, B) \to  \Ho^{1}(L, B) \to \ldots
		$$
		$$ \to  \Ho^i(L,B ) \to \oplus_{v \in V(\Gamma)} \Ho^i(L_v, B) \to  \oplus_{  e \in E(\Gamma)}  \Ho^i(L_e, B) \to  \oplus \Ho^{i+1}(L, B) \to \ldots$$	
	\end{corollary}
	
	 A Lie algebra $L$ is of type $FP_n$ if the trivial $U(L)$-module $k$ is $FP_n$ i.e. there is a projective resolution of the module where all projectives are finitely generated in dimensions $\leq n$. We note that type $FP_1$ is equivalent to $L$ being finitely generated as a Lie algebra. If $L$ is finitely presented in terms of generators and relations then $L$ is $FP_2$. Whether the converse holds is an open problem.
	
\begin{lemma}\label{indFPn} Let $S\leq L$ be a Lie subalgebra of an arbitrary Lie algebra $L$. Then $S$ is of type $\FP_n$ if and only if the induced $U(L)$-module $k \otimes_{U(S)} U(L)$ is of type $\FP_n$. 
\end{lemma}
\begin{proof}  If $S$ is $\FP_n$ there is a  resolution $\mathcal{P}$ of the trivial $U(S)$-module  $k$ with finitely generated modules in dimensions $ \leq n$. Then ${\mathcal P} \otimes_{U(S)} U(L)$ is a resolution that shows that $k \otimes_{U(S)} U(L)$ is $\FP_n$ as $U(L)$-module. 

For the converse suppose we use induction on $n \geq 1$, the case $n = 1$ is easy. For the inductive step suppose that $n > 1$ and that the result holds for modules of type $FP_{n-1}$ and suppose that $k \otimes_{U(S)} U(L)$ is $FP_n$. By induction $S$ is $FP_{n-1}$. Then we have an exact complex of $U(S)$-modules
$${\mathcal P} : 0 \to Ker (d_n) \to P_{n-1} \Vightarrow{d_n} P_{n-2} \to \ldots \to P_0 \to k \to 0$$
where $k$ is the trivial $U(S)$-module and $P_i$ is projective and finitely generated for $i \leq n-1$. Then we have an exact complex
$${\mathcal P} \otimes_{U(S)} U(L) : 0 \to Ker (d_n) \otimes_{U(S)} U(L) \to P_{n-1} \otimes_{U(S)} U(L) \Vightarrow{d_n} P_{n-2}  \otimes_{U(S)} U(L) \to$$ $$  \ldots \to P_0 \otimes_{U(S)} U(L) \to k \otimes_{U(S)} U(L) \to 0$$
Since each $P_i \otimes_{U(S)} U(L)$ is a finitely generated projective module and $k \otimes_{U(S)} U(L)$ is $FP_n$ as $U(L)$-module we deduce by \cite[Prop.~4.3]{Brown} that $Ker (d_n) \otimes_{U(S)} U(L)$ is a fintely generated $U(L)$-module, hence $Ker(d_n)$ is finitely generated as $U(S)$-module and so $S$ is $FP_n$.

\end{proof}

\begin{corollary}
	 Let $\Delta$ be a graph of Lie algebras with fundamental Lie algebra $L = L_{\Delta}$.
	Then
	
	a) if  $L_e$ is $\FP_{m-1}$ for every $e \in E(\Gamma)$ and $L_v$ is  $\FP_{m}$ for every $v \in V(\Gamma)$ then $L$ is $\FP_{m}$; 
	
	b) if $L$ is $\FP_{m}$ and $L_e$ is $\FP_{m}$ for every $e \in E(\Gamma)$
	then $L_v$ is  $\FP_{m}$ for every $v \in V(\Gamma)$;
	
	c) if  $L$ is $\FP_{m}$ and $L_v$ is $\FP_{m}$ for every $v \in V(\Gamma)$ then $L_e$ is $\FP_{m-1}$ for every $e \in E(\Gamma)$.	
\end{corollary}

\begin{proof} Let $R$ be any associative ring with 1 and  $0 \to A \to B \to C \to 0$ be a short exact sequence of $R$-modules. Then by \cite[Proposition 1.4]{Bieribook}
	
	1) if $A$ and $C$ are $\FP_m$ then $B$ is $\FP_m$;
	
	2) if $A$ is $\FP_{m-1}$ and $B$ is $\FP_m$ then $C$ is $\FP_m$;
	
	3) if $B$ and $C$ are $\FP_m$ then $A$ is $\FP_{m-1}$.
	
The Corollary is a consequence of the above statements applied for $R = U(L)$ together with Lemma \ref{indFPn} and the short exact sequence given by Theorem \ref{exact1}.

\end{proof}

  The following two results are well known for groups and extend also to Lie algebras.

  \begin{proposition}  \label{free0} \label{free-bass-serre} a) \cite{L-S} Let $L$ be a Lie algebra HNN extension with a base Lie subalgebra $L_1$, associated Lie subalgebra $L_0$ and stable letter $t$. Let $H$ be a Lie subalgebra of $L$ such that $H \cap L_0 = 0$ and $H \cap L_1$ is a free Lie algebra. Then $H$ is a free Lie algebra.
 	
 	b) \cite{Kukin} Let $L = L_1 *_{L_0} L_2$ be  an amalgameted product of Lie algebras and $H$ be a Lie subalgebra of $L$  such that $H \cap L_0 = 0$ and $H \cap L_i$ is free for $i = 1,2$. Then $H$ is a free Lie algebra.
\end{proposition}

 As a consequence we have

\begin{theorem} \label{free} 
 Let $\Delta$ be a graph of Lie algebras with a fundamental Lie algebra $L = L_{\Delta}$. If $H$ intersects every edge Lie algebra $L_e$ trivially  and every vertex Lie algebra in a free Lie subalgebra then $H$ is a free Lie algebra.
 \end{theorem}

\begin{proof} It follows by induction on the number of vertices of the underlying graph using for the inductive step the previous proposition.
	
	\end{proof}

Note that the above theorem cannot be proved by cohomological reasoning since by \cite{M-U-Z} if $\mathrm{char}(k) = p\geq 3$ there is a Lie algebra such that $U(L)$ is a free associative $k$-algebra, hence the projective dimension of the trivial $U(L)$-module $k$ is 1 but $L$ is not free. Furthermore  in \cite{Shirshov2} Shirshov showed the existence of Lie algebras such that their free Lie product has a Lie subalgebra that is not free, is not isomorphic to any subalgebra of any of the factors and cannot be decomposed as the free Lie product of any of its subalgebras. Thus we cannot have a very general Kurosh type decomposition theorem generalising  Theorem \ref{free}.

    \subsection{An application: a resolution of the trivial module for right angled Artin Lie algebras}\label{RAAGs}

Let $\Gamma$ be a finite simple graph with vertex set $V(\Gamma)$ and edge set $E(\Gamma)$, here simple means that there are no loops or double edges and we see each edge as an unordered pair $\{v,w\}$ with $v,w\in V(\Gamma)$. Associated to $\Gamma$ there is  a group called the {\it right angled Artin group} $G_\Gamma$ given by the presentation by generators and relations in the category of groups
$$G_\Gamma=\langle V(\Gamma)\mid [u,v]=1\text{ whenever }\{u,v\}\in E(\Gamma)\rangle$$
and also a $k$-Lie algebra  called the {\it right angled Artin Lie algebra $L_\Gamma$} given by the presentation by generators and relations in the category of Lie algebras
$$L_\Gamma=\langle V(\Gamma)\mid [u,v]=0\text{ whenever }\{u,v\}\in E(\Gamma)\rangle.$$
Recall that the flag complex $\Delta_{\Gamma}$  associated to $\Gamma$ is the  simplicial complex with simplices the (non-empty) ordered subsets of vertices of $\Gamma$ that span a complete subgraph. The cone $C\Delta_\Gamma$ is the complex that one gets allowing the subsets of vertices to be empty. We can linearly order the vertices in $V(\Gamma)$.  Consider the following complex 
\begin{equation}\label{complexArtin}
{\mathcal P}_{\Gamma} :  \ldots \Vightarrow{\partial_{n+1}} P_n \Vightarrow{\partial_n} P_{n-1} \Vightarrow{\partial_{n-1}} \ldots \Vightarrow{\partial_1} P_0 \Vightarrow{\partial_{0}}  k \to 0,\end{equation}
where 
$$P_n = \oplus_{w = \{ v_{i_1}, \ldots, v_{i_n} \} } c_{w} U(L)$$
and the direct sum is over the simplices of $C\Delta_\Gamma$, i.e., over all (possibly empty) subsets $w = \{ v_{i_1}, \ldots, v_{i_n} \} $ of $V(\Gamma)$ such that the subgraph of $\Gamma$ that they span is complete. Every $c_{w} U(L)$ is isomorphic to the free right $U(L)$-module and we have  for the differential after assuming $ v_{i_1} <  \ldots < v_{i_n}$ that
$$
\partial_n(c_w) =  \sum_r (-1)^{r-1} c_{w \setminus \{ v_{i_r} \} } v_{i_r},
$$
 $c_{\emptyset} = 1_K$, $P_0 = c_{\emptyset} U(L) = U(L)$ and $\partial_0$ is the augmentation map.
So (\ref{complexArtin})  is a Lie algebra version of the projective resolution associated to the universal cover of the Salvetti complex for right angled Artin groups. This complex  (\ref{complexArtin}) is the {\sl minimal resolution} of the Lie algebra $L_\Gamma$ and in particular is exact  (see \cite[proof of Theorem 1.2, pages 11 and 12]{BRS}). This fact can be understood as a Lie algebra version of the fact that for groups the Salvetti complex is a $K(G_{\Gamma}, 1)$. In the next result we derive it using the short exact sequence for a free product with amalgamation of Lie algebras,  note that our argument is different than the one in \cite{BRS} and was obtained independently.

\begin{proposition} \label{free-complex}
${\mathcal P}_{\Gamma}$ is a free resolution of the trivial $U(L)$-module $K$.
\end{proposition}

\begin{proof} We induct on the number of vertices in $\Gamma$.
 We consider first the case when $\Gamma$ is a full graph. Then $L=L_\Gamma$ is abelian and ${\mathcal P}$ is the standard Koszul complex, so is exact. 

Assume now that $\Gamma$ is not complete. Then we can find full proper subcomplexes $\Gamma_1,\Gamma_2,\Gamma_0$ such that $\Gamma_1 \cap \Gamma_2 = \Gamma_0$ and  $\Gamma =  \Gamma_1 \cup \Gamma_2$. The presentation of $L_i=L_{\Gamma_i}$, $i=0,1,2$ in terms of generators and relations gives that
$$
L = L_1 *_{L_0} L_2.$$
 Then by induction
${\mathcal P}_{\Gamma_i}$ are all exact complexes for $i = 0, 1, 2$.

By Proposition \ref{amalgam*}  there is an exact sequence of right $U(L)$-modules
$$
0 \to k \otimes_{U(L_0)} U(L) \to (k \otimes_{U(L_1)} U(L)) \oplus (k \otimes_{U(L_2)} U(L)) \to k \to 0
$$
where the first maps sends $k_0 \otimes \lambda$ to   $(k_0 \otimes \lambda) + (k_0 \otimes - \lambda)$ and the second sends $k_1 \otimes \lambda_1 + 
k_2 \otimes \lambda_2 $ to $k_1 \epsilon_1(\lambda_1) + k_2 \epsilon_2(\lambda_2)$, where $\epsilon_i : U(L_i) \to k$ is the augmentation map. This short exact sequence extends to the short exact sequence of complexes
\begin{equation} \label{complex} 
0 \to {\mathcal P}_{\Gamma_0} \otimes_{U(L_0)} U(L) \to
 ({\mathcal P}_{\Gamma_1} \otimes_{U(L_1)} U(L)) \oplus 
({\mathcal P}_{\Gamma_2} \otimes_{U(L_2)} U(L)) \to {\mathcal P}_{\Gamma}  \to 0
\end{equation} 
which is induced by
$c_w \otimes 1$ going to $c_w \otimes 1 + ( c_w \otimes -1)$ and
$c_{w_1} \otimes \lambda_1 + c_{w_2} \otimes \lambda_2 $ goes to $c_{w_1} \lambda_1 - c_{w_2} \lambda_2$ i.e. it is induced by the natural embeddings of ${\mathcal P}_{\Gamma_0}$ in ${\mathcal P}_{\Gamma_i}$  and of ${\mathcal P}_{\Gamma_i}$ in ${\mathcal P}_{\Gamma}$ for $i = 1,2$. 

By induction the complexes ${\mathcal P}_{\Gamma_i}$ are exact for $i = 0,1,2$ hence   the complexes ${\mathcal P}_{\Gamma_i} \otimes_{U(L_i)} U(L)$ are exact for $i = 0,1,2$. Thus (\ref{complex}) is a short exact sequence of complexes, where the first two of the complexes are exact hence the last one, ${\mathcal P}_{\Gamma}$, is exact too.

\end{proof}

\section{Graded Lie algebras}

Here we recall some results on graded Lie algebras and their homological properties.

  We denote $\mathbb{N} = \{ 1,2, \ldots \}$  and $\mathbb{N}_0 = \mathbb{N} \cup \{ 0 \}$. An $\mathbb{N}$-graded Lie algebra  is a Lie algebra  $L$ with a decomposition as vector space $$L = \oplus_{i \in \mathbb{N}}  L_i, \hbox{ where }[L_i, L_j] \subseteq L_{i+j}.$$  $L$ is called  an $\mathbb{N}$-graded Lie algebra of {\it finite} type if each $L_i$ is finite dimensional  over $k$. Note that every finitely generated $\mathbb{N}$-graded Lie algebra $L$ is of finite type. It  is easy to check whether an $\mathbb{N}$-graded Lie algebra $L$ is finitely generated, as it is equivalent to $L/ [L, L] \simeq H_1(L, k)$ is finite dimensional.
  
The $\mathbb{N}$-grading on $L$  induces  a natural $\mathbb{N}_0$-grading on the universal enveloping algebra $R = U(L) $ of $L$ defined by $$ R = \oplus_{i \in \mathbb{N}_0} R_i, \hbox{ where }R_0 = K.1 \hbox{ and }R_i = \sum_{i_1 + \ldots + i_j = i, j \geq 1   }L_{i_1} \ldots L_{i_j}$$ An $\mathbb{N}_0$-graded $R$-module $V = \oplus_{i \in \mathbb{N}_0} V_i$ is defined by the property $V_i R_j \subseteq V_{i+j}$.  A homomorphism of $R$-modules 
$$\varphi : V = \oplus_{i \geq 0} V_i \to W = \oplus_{i \geq 0} W_i$$ between $\mathbb{N}_0$-graded $R$-modules is called graded if $\varphi(V_i) \subseteq W_i$ for every $i$.

We state the version of the Nakayama Lemma for an arbitrary $\mathbb{N}_0$-graded $R$-modules $V$: $V = 0$ if and only if $V \otimes_{R} k = 0$ \cite{Weigel}.
 Furthermore, $V$ is finitely generated as $R$-module if and only if $V \otimes_{R} k$ is finite dimensional over $k$ \cite[Section 2]{Weigel}. 
 
 \begin{lemma} \label{neu}  \cite{K-M2}, \cite{Weigel} Let $L$ be an $\mathbb{N}$-graded Lie algebra. Then $L$ is of type $FP_2$ if and only if $L$ is finitely presented (in terms of generators and relations).
 	\end{lemma}

Let $L$ be an $\mathbb{N}$-graded Lie algebra. We say that $L$ is {\it graded} $FP_m$ if there is a graded projective resolution ( i.e. of graded $U(L)$-modules with graded homomorphisms) of the trivial $U(L)$-module  $k$, such that the projective modules in dimension smaller or equal to $m$ are all finitely generated.

Let $L = \oplus_{i \geq 1} L_i$ be an $\mathbb{N}$-graded Lie algebra. An ideal $N$ of $L$ is a {\it graded} ideal if $N = \oplus_{i \geq 1} N \cap L_i$.

The following result shows that homologically graded Lie algebras behave as pro-$p$ groups i.e. it is a Lie graded algebra version of the pro-$p$-groups result \cite[Thm.~A]{King}. 

\begin{proposition}  \label{homgrad2} \cite{K-M2} Let $L$ be an $\mathbb{N}$-graded Lie algebra over a field $k$ and $N$ be a graded ideal such that $U(L/ N)$ is left and right Noetherian. Then
	$L$ is graded $FP_m$ if and only if $\Ho_i(N, k)$ is finitely generated as $U(L/ N)$-module for every $i \leq m$.
This implies that both graded $FP_m$ and ordinary $FP_m$ are the same property.	\end{proposition} 

Applying the above result for $N = L$ we obtain that $L$ is $FP_m$ if and only if $\Ho_i(L,k)$ is finite dimensional for every $i \leq m$.

For completeness we include two more results, that probably are well known. In both cases the results show how similar $\mathbb{N}$-graded Lie algebras are to pro-$p$ groups, since  the same type of results hold for pro-$p$ groups with $k$ substituted  with $\mathbb{F}_p$.

\begin{lemma} Let $L$ be a $\mathbb{N}$-graded Lie algebra. Then $L$ is a free Lie algebra if and only if $\Ho_2(L,k) = 0$.
\end{lemma}

\begin{proof} 
Let $X$ be a minimal set of generators of $L$. As $L$ is $\mathbb{N}$-graded we have that $L = [L,L] + k X$, where $kX$ is the $k$-vector space spanned by $X$. Since $L$ is $\mathbb{N}$-graded we can choose $X$ to be a graded subset of $L$ i.e. $X $ is a disjoint union $\cup_{ i \geq 1} X_i$ where $X_i \subseteq L_i$.

Let $F(X)$ be the free  Lie algebra on $X$ and $F(X) / R \simeq L$. Note that we have a natural $\mathbb{N}$-grading  of $F(X)$, where the elements of $X_i$ have degree $i$, thus $F(X) / R \simeq L$ is an isomorphism of $\mathbb{N}$-graded Lie algebras.
Since $L$ is $\mathbb{N}$-graded and $X$ is minimal, we have that $dim_k L/ [L,L] = | X | = dim_k F(X)/ [F(X), F(X)]$, hence $R \subseteq [F,F]$.  By Hopf formula $$0 = H_2(L,k) \simeq (R \cap[F,F])/ [R, F] = R/ [R, F] \simeq R/ [R, R] \otimes_{U(L)} k,$$
where we view $R/ [R, R]$ as a right $U(L)$-module via the adjoint action of $F(X)$ that factors through $F(X)/ R \simeq L$.
By the Nakayama lemma for the graded $U(L)$-module $R/[R,R]$ we get that $R/ [R, R] = 0$. Since every subalgebra of a free Lie algebra is free (\cite[Theorem 2.8.3]{BokutKukin}) we conclude that $R$ is free, hence $dim_k (R/ [R,R])$ is the minimal number of generators of $R$. Thus $R = 0$ and $L \simeq F(X)$ is free.
\end{proof}

Let $L$ be an $\mathbb{N}$-graded Lie algebra. An $\mathbb{N}$-graded presentation $\langle X ~| ~R_0 \rangle$ is a presentation, where $X \subseteq L$ is an $\mathbb{N}$-graded  subset of $L$ and $R_0$ is $\mathbb{N}$-graded  subset of the $\mathbb{N}$-graded  free Lie algebra $F(X)$ with a free basis $X$. Note that if $L$ is an $\mathbb{N}$-graded Lie algebra with an $\mathbb{N}$-graded generating set $X$ then $L$ has an $\mathbb{N}$-graded presentation $\langle X ~| ~R_0 \rangle$.

\begin{lemma} \label{presentation} Let $L$ be a $\mathbb{N}$-graded Lie algebra with $\mathbb{N}$-graded presentation $\langle X ~| ~R_0 \rangle$, with $X$ minimal and $R_0$ minimal possible once $X$ is fixed. Then
$$|X| =  dim_k H_1(L, k), |R_0| = dim_k H_2(L,k)$$
\end{lemma}

\begin{proof} The first follows immediately from the fact that by the minimality of $X$ the image of $X$ in $H_1(L,k) \simeq L/ [L,L]$ is a basis as a $k$-vector space. Set $R$ the  kernel of the epimorphism $F(X) \to G$ that is the identity on $X$, i.e. we have $L \simeq F(X)/ R$ and 
$$H_2(L,k) \simeq (R \cap[F,F])/ [R, F] = R/ [R, F] \simeq R/ [R, R] \otimes_{U(L)} k.$$
Note that $R_0$ generates $R$ as an ideal of $F(X)$, hence the image of $R_0$ is a generating set of $R/ [R,R]$ as $U(L)$-module, where we view $R/ [R,R]$ as $U(L)$ via the adjoint action. Then the image of $R_0$ is a generating set of $R/ [R,R] \otimes_{U(L)} k$ as a $k$-vector space. Let $R_1$ be a subset of $R_0$ that is a basis of $R/ [R,R] \otimes_{U(L)} k$ as a vector space and let $I$ be the ideal of $F(X)$ generated by $R_1$. Then $R = I + [R,R]$ and both $I$ and $R$ are $\mathbb{N}$-graded ideals ( hence graded Lie subalgebras of $F(X)$). Hence for the $\mathbb{N}$-graded Lie algebra $S = R/ I$ we have $S = [S, S]$, hence $S = 0$, i.e. $R = I$.
\end{proof}

\section{Graded one relator Lie algebras as iterated HNN extensions}\label{iterated}

A 1-relator Lie algebra is a Lie algebra admitting a presentation (in terms of generators and relations)  of the form
 $$L=\langle X\mid r\rangle$$
where $X$ is an arbitrary set and $r$ is an element of the free Lie algebra $F(X)$ on $X$. In other words, $L$ is the quotient of $F(X)$ with the ideal generated by $r$.   Assume that there is some weight function $\omega:X\to\Z^+$. This induces an $\mathbb{N}$-grading on $F(X)$. 
If $r$ is homogeneous when seen as element of $F(X)$ then the ideal that it generates is also homogeneous so the quotient algebra $L$ is $\mathbb{N}$-graded.

In this section we use the notation $A *_h$ for a Lie algebra HNN extension with a base subalgebra $A$ and stable letter $t$. This notation does not specify the associated Lie subalgebra. We prove the following result which is Theorem B from the introduction.

\begin{theorem}\label{onerelator} Let 
$L=\langle X\mid r\rangle$
be a one relator  $\mathbb{N}$-graded  Lie algebra with $X$ finite, where  the $\mathbb{N}$-grading is induced by  some weight function $\omega:X\to\Z^+$, i.e.  $r\in F(X)$ is homogeneous with respect to $\omega$.  Then $L$ is an iterated HNN-extension  of Lie algebras
$$L=(\ldots(A*_{h_n})*_{h_{n-1}})\ldots)*_{h_1}$$
such that $A$ and all the associated  Lie subalgebras in each of the HNN-extensions are free.
 \end{theorem}
 
As a consequence one can prove that for $L$ as in the hypothesis $\cd(L) \leq 2$ (but this is already proven for arbitrary one-relator Lie algebras in \cite[3.10.7]{BokutKukin}).
 The proof of Theorem \ref{onerelator} uses the Freihitssatz for one-relator Lie algebras \cite{BokutKukin}
and also an inductive argument based on the proof of  the following result by Labute  \cite{Labute95}:
 
 \begin{theorem}(Labute, \cite[ Proof of Theorem 1, page 182]{Labute95})\label{Labute} Let 
$$L=\langle X\mid r\rangle$$
be a 1-relator Lie algebra with $X$ finite. Assume that there is some weight function $\omega:X\to\Z^+$ such that $r$ is homogeneous with respect to the grading induced by $\omega$ on $F(X)$. Then there is  some ideal $\mathcal{H}$ of $L$ of finite codimension such that $\mathcal{H}$ is free as Lie algebra.
\end{theorem}

Note first that we may assume $|X|\geq 1$. The weight function $\omega$ can be extended to $F(X)$ via $\omega([u,v])=\omega(u)+\omega(v)$.
Fix an order $<$ in $X$ which is compatible with $\omega$ (i.e., such that $x<y$ implies $\omega(x)\leq\omega(y)$). 
Following  \cite{Labute95} a {\emph weighted Hall set} with respect to  $X$ and $\omega$ is a basis $H$ of $F(X)$ consisting of Lie monomials together with a well-ordering $<$ such that 
\begin{itemize}
\item[1)] $X$ is an ordered subset of $H$,

\item[2)] for $u,v\in H$, $\omega(u)<\omega(v)$ implies $u<v$,

\item[3)] for $[a,b]\in F(X)$, $[a,b]\in H$ if and only if the following conditions hold: $a,b\in H$, $a<b$ and if $b=[c,d]$,   $c,d \in H$ then $a\geq c$,

\item[4)] if $[a,b],[c,d]\in H$ are such that $\omega([a,b])=\omega([c,d])$, then $[a,b]<[c,d]$ if and only if either $b<d$ or $b=d$ and $a<c$. 
\end{itemize}
From this it can be deduced that every $v\in H$ admits a \emph{canonical decomposition} which is an expression
$$v=[u_1,\ldots,u_m,z]$$
(commutators are right normed) with $u_1\geq\ldots\geq u_m<z$, $u_1\ldots,u_m\in H$, $z\in X$. 

\noindent{\it Proof of Theorem \ref{Labute} (Sketch)} Consider $H$ with the total order induced by $<$:
$$H=\{h_1<h_2<\ldots\}$$
and define $\Ho_0 = H$, $X_0 = X$ and
$$\Ho_i=\Ho_{i-1}  \setminus \{h_i\}=H \setminus\{h_1,\ldots,h_i\},$$
$$X_i=\{[ \underbrace{h_i,\ldots,h_i}_{n\text{-times} },z]\mid  n \geq 0 , h_i\neq z\in X_{i-1}\}.$$
Then $X_i$ is the set of indecomposable elements of $\Ho_i$ and $h_{i+1}\in X_i$. Also,
$$X_{i-1} \setminus
\{h_i\}\subseteq X_i.$$
Moreover, $\Ho_i$ is a weighted Hall set with respect to $X_i$ and the weight induced by $\omega$. In particular, $\Ho_i$ is a basis for the free Lie algebra $F(X_i)$. We also have
$$F(X)=F(X_0)\supseteq F(X_1)\supseteq\ldots\supseteq F(X_i)\supseteq\ldots$$
and each $F(X_{i+1})$ is an ideal of $F(X_i)$ of codimension 1 such that $F(X_i)/F(X_{i+1})=Kh_{  i+1}$. Finally, $\cap F(X_i)=0$.

As $k$ is a field, we may  assume that  $$r=a+\alpha$$
where $a\in H$ and $\alpha$ is a linear combination of elements $\alpha_j\in H$ with $\alpha_j<a$.
Let $I$ denote the ideal of $F(X)$ generated by $r$.
Labute shows that for any $i$, if $I\subseteq F(X_i)$ then
\begin{itemize}
\item either $I\subseteq F(X_{i+1})$ and $a$ and each $\alpha_j\in \Ho_{i+1}$,
\item or $a\in X_i$, the family $\mathcal{F}$ below generates the ideal $I$ as an ideal of $F(X_i)$ and $\mathcal{F}$ can be extended to a basis of $F(X_i)$. In this case, $F(X_i)/I$ is free.
$$\mathcal{F}=\{[h_i,\buildrel{j_i}\over{\ldots} , h_i,h_{i-1},\buildrel{j_{i-1}}\over{\ldots},h_{i-1},\ldots,h_1,\buildrel{j_1}\over\ldots,h_1,r]\mid j_i,j_{i-1},\ldots,j_1\geq 0\}.$$
\end{itemize}
As $\cap F(X_i)=0$ this process has to stop eventually so there is some $t$ with $I\subseteq F(X_t)$ and $\mathcal{H}:=F(X_t)/I$ is free. Moreover $\mathcal{H}$ is an ideal of $L$ and
$$L/\mathcal{H}=F(X)/F(X_t)$$
has finite dimension.
\qed
 
 \bigskip

 \begin{remark}\label{epiHNN} If $L=B*_h$ is an  HNN extension of $k$-Lie algebras, there is an epimorphism $\pi:L\to k$ with $\pi(h)=1$. The kernel of $\pi$ is the ideal of $L$ generated by the Lie subalgebra $B$. 
  \end{remark}

 \noindent{\it Proof of Theorem \ref{onerelator}} We use the same notation as in the proof of Theorem \ref{Labute}. We begin by choosing, for any $i\leq t$, a 
  finite subset
 $$Y_i\subseteq X_i$$ 
 such that $r\in F(Y_i)$, $h_i\in Y_{i-1}$ and $Y_{i-1} \setminus \{h_i\}\subseteq Y_i$. To do that,    let $Y_0=X_0=X$  and proceed inductively:  if $r\in F(Y_{i-1})$ we can choose a non-negative integer $j(i)$ such that $r$ lies in the free Lie subalgebra generated by
 $$Y_i=\{[h_i,\buildrel j\over\ldots,h_i,z]\mid z\in Y_{i-1} \setminus    \{h_i\},0\leq j\leq j(i)\}.$$
 Observe that at each step we could be choosing  the integer $j(i)$ to be smallest possible, however we prefer not to do so at this point and allow $j(i)$ to be arbitrarily big. For each $i$ let
  $$Z_i=\{[h_i,\buildrel j\over\ldots,h_i,z]\mid z\in Y_{i-1} \setminus    \{h_i\},0\leq j< j(i)\}.$$

  The adjoint action of $h_i$ induces a map
  $$
  \begin{aligned}
  Z_i&\to Y_i\\
  b&\mapsto[h_i,b]\\
    \end{aligned}
  $$
  Put $B_i=\langle Y_i\mid r\rangle$ and let $A_i$ be the Lie subalgebra of $B_i$ generated by $Z_i$. The map above induces a derivation
  $$
  \begin{aligned}
  d_i:A_i&\to B_i\\
  b&\mapsto[h_i,b]\\
    \end{aligned}
  $$    
and we may form the Lie algebra HNN extension $B_i *_{h_i}$ with associated subalgebra $A_i$. 
We have
$$B_i*_{h_i}=\langle Y_i,h_i\mid r,[h_i,b]=d(b)\text{ for }b\in A_i\rangle=\langle Y_{i-1}\mid r\rangle=B_{i-1}.$$
 Therefore
  $$L=F(X)/I=(\ldots((B_t*_{h_t})*_{h_{t-1}})\ldots*_{h_1}$$
and $B_t=\langle Y_t\mid r\rangle$. Here $t$ is the number from the  scketch of the proof of Theorem \ref{Labute} and $B_t$ is a Lie subalgebra of the free Lie algebra ${\mathcal{H}}$, hence is free itself. 

At this point we have shown that $L$ is an iterated HNN-extension of one relator algebras so that the first one is free and we still have to prove that the associated algebras $A_i$ can be chosen to be all free. To do that it suffices to chose all the values  $j(i)$ to be minimal possible. Then we have that $r\not\in F(Z_i)$ for any $i$ which, using the Freiheitssatz for Lie algebras implies that the  Lie algebras $A_i$ are free. 
 \qed

\section{Coherence for Lie algebras and universal enveloping algebras}\label{coherence}

Let $R$ be an associative ring with 1. Recall that an $R$-module $M$ is called  coherent if every finitely generated $R$-submodule of $M$ is finitely presented.
The ring $R$  is called {\it coherent} (meaning {\it right coherent} with our convention) if it has the property that any finitely generated   $R$-submodule of $R$, i.e. right ideal,  is finitely presented. This is equivalent to every finitely presented (right) $R$-module is coherent. Note that a finitely generated coherent $R$-module $M$ is of homological type $FP_{\infty}$, in particular is finitely presented  and that every finitely presented $R$-module over a coherent ring $R$ is a coherent $R$-module, hence is $FP_{\infty}$. For example,  free non-abelian polynomial rings over a field are coherent  \cite[Cor.~3.3]{Swan}.  And abelian polynomial rings over a filed and  on infinitely many variables are coherent but  not noetherian.  And a group $G$ is {\it coherent} if  any finitely generated subgroup $H\leq G$ is also finitely presented. Again, free groups are coherent. By analogy we set

\begin{definition} A Lie algebra is {\it coherent} if any finitely generated subalgebra $S\leq L$ is also finitely presented. If $L$ is $\N$-graded, then we say that it is {\it graded-coherent} if
any finitely generated $\N$-graded subalgebra $S\leq L$ is also finitely presented.  
\end{definition}

Assume that  for a group $G$ the group ring $kG$ is coherent and $H$ is a finitely generated subgroup of $G$. As $H$ is finitely generated, the induced module $U=k\otimes_{kH}kG$ is a finitely presented (right) $k G$-module, hence is  $\FP_{\infty}$ and in particular is $FP_2$. This implies that the group $H$ is $FP_2$ by \cite[Prop.~ 1.4]{Bieribook}, see Lemma 
\ref{indFPn}.  
 But in general this does not imply that $H$ is finitely presented,  so we cannot claim that $G$ is a coherent group. The problem is that, for groups, being of type $\FP_2$ does not imply being finitely presented as the examples constructed by Bestvina and Brady in \cite{B-B} 
show. However, things change if we work with $\N$-graded Lie algebras, as in that case both properties are equivalent (see \cite{Weigel} and also \cite{K-M2} where we extend this fact to higher degrees).

\begin{lemma}\label{graded} Let $L$ be an $\N$-graded Lie algebra. Assume that the universal enveloping algebra $U(L)$ is coherent. Then $L$ is graded-coherent.
\end{lemma}
\begin{proof} Let $S\leq L$ be a finitely generated $\N$-graded subalgebra. Consider the induced $U(L)$-module $M=k\otimes_{U(S)}U(L)$. By Lemma \ref{indFPn}, $M$ is $\FP_1$. As $U(L)$ is coherent $M$ is also of type $\FP_2$ and so is $S$ again by Lemma \ref{indFPn}. As $S$ is $\N$-graded, $S$ is finitely presented (see above).
\end{proof}

For rings there is a construction of free amalgamated products and also a notion of HNN extension. In both cases, if the rings that play the role of the edge groups are assumed to be Noetherian and the rings that play the role of the vertex groups are assumed to be coherent, the resulting ring is coherent (see \cite{A} for free amalgamated products, \cite{Dicks} for HNN extensions). 

At this point, we can prove Theorem C of the introduction.

\begin{theorem}\label{coherentfund}
Assume that the Lie algebra $L$ is the fundamental Lie algebra of a graph of Lie algebras such that for the vertex Lie algebras $L_v$ the universal enveloping algebra $U(L_v)$ is coherent and for each edge algebra $L_e$, $U(L_e)$ is Noetherian. Then $U(L)$ is coherent.
\end{theorem}

\begin{proof}  The result follows by an adaptation of the main argument of \cite[Corollary 13]{A} using the long exact sequence of Tor functors above. For the reader's convenience, we summarize the idea of the proof  and shift from left coherence considered in \cite{A} to right coherence considered here. The main ingredients are:
\begin{itemize}
\item[i)]  By  a right-module version of Lemma 7 in \cite{A} (the result goes back to Chase), a ring $R$ is  right coherent if and only if $\prod_IR$ is   flat as left $R$-module, where $I$ is a set of cardinality $\mathrm{card} R$. And this is equivalent to any   module of the form $\prod_{i\in J}E_i$ being flat, where $J$ is an arbitrary index set and the $E_i$ are flat  left $R$-modules.

\item[ii)] If a ring $R$ is Noetherian, then it is coherent and moreover the natural map $(\prod_{i\in J}B_i)\otimes_RA\to\prod_{i\in J}(B_i\otimes_RA)$ is an injection where $J$ is any index set and $B_i,A$ are arbitrary $R$-modules.
\end{itemize}
Now, let  $V$ be any right $R$-module for $R=U(L)$ and $I$ an index set of cardinality $\mathrm{card}R$.  By the long exact sequence of Tor functors  Corollary \ref{ses-homology} for any $s\geq 1$,  $\Tor_s^L(V,\prod_IR)$ is sandwiched as follows:
$$\to \oplus_{v \in V(\Gamma)}\Tor_s^{U(L_v)}(V,\prod_I R)\to\Tor_s^R(V,\prod_IR)\to \oplus_{e \in E(\Gamma)}\Tor_{s-1}^{U(L_e)}(V,\prod_IR)\to$$
Now, as each $U(L_v)$ is coherent, i) implies  $\oplus_{v \in V(\Gamma)}\Tor_s^{U(L_v)}(V,\prod_IR)=0$  for $s \geq 1$. If $s>1$,    i) implies  that also $\oplus_{e \in E(\Gamma)}\Tor_{s-1}^{U(L_e)}(V,\prod_IR)=0$. This implies $\Tor_s^R(V,\prod_IR)=0$ for $s>1$. In the case when $s=1$ we have a commutative diagram with exact rows

$$\begin{tikzcd}
0\arrow[r]&\Tor_1^R(V,\prod_IR)\arrow[r]\arrow[d]&\oplus_{e \in E(\Gamma)}V\otimes_{ U(L_e)}(\prod_IR)\arrow[d,"\tau"]\arrow[r,"\mu"]&\oplus_{v \in V(\Gamma)}V\otimes_{ U(L_v)}(\prod_IR)\arrow[d]\\
&0  = \prod_I \Tor_1^R(V,R) \arrow[r]&\oplus_{e \in E(\Gamma)}\prod_I(V\otimes_{L_e}R)\arrow[r]&\oplus_{v \in V(\Gamma)}\prod_I(V\otimes_{L_v}R)\\
\end{tikzcd}$$
As $\tau$ is a monomorphism by ii), we deduce that also $\mu$ is mono thus  $\Tor_1(V,\prod_IR)=0$. As this happens for any $V$, we deduce that $\prod_IR$ is flat  as left module so i) implies that $R$ is coherent.
\end{proof}

Recall that groups of the form $\langle X\cup Y\mid u=v\rangle$ where $X$ and $Y$ are disjoint, $u$ is a word in $X$ and $v$ a word in $Y$ are called {\it one relator pinched} and groups  $\langle X,t\mid t^{-1}ut=v\rangle$ where $t\not\in X$, $u$ and $v$ are words in $X$ are called {\it one relator cyclically pinched}. By analogy, we set

\begin{definition}  A Lie algebra of the form $\langle X\cup Y\mid u=v\rangle$ where $X$ and $Y$ are disjoint, $u$ lies in the free Lie algebra generated by $X$ and $v$ lies in the free Lie algebra generated by $Y$ is called {\it one relator pinched} and a Lie algebra  $\langle X,t\mid  [u,t] =v\rangle$ where $t\not\in X$ and $u$ and $v$ lie in the free Lie algebra generated by $X$ are called {\it one relator cyclically pinched}. 
\end{definition}
 
 \begin{corollary} One relator pinched and one relator cyclically pinched $\mathbb{N}$-graded Lie algebras such that the corresponding relators are homogeneous are graded-coherent.
 \end{corollary}
 \begin{proof} In both cases, the corresponding Lie algebra is $\N$-graded (because the relators are assumed to be homogeneous). Also, both constructions yield the fundamental Lie algebra of a graph of Lie algebras with free Lie algebras as vertices and  one dimensional Lie algebra as an edge.  For a free Lie algebra $F$, $U(F)$ is the free $K$-polynomial ring  which is coherent and for a one-dimensional Lie algebra $T$, $U(T)$ is the polynomial ring in one variable, which is Noetherian. It suffices to use Lemma \ref{graded} and Theorem \ref{coherentfund}.
 \end{proof}
 
In the case of a right angled Artin group $A_\Gamma$, $\Gamma$ a  finite connected graph, it was proven by Droms \cite{D} that $A_\Gamma$ is coherent if and only if the graph $\Gamma$ is {\it chordal}, meaning that there is no $k$-cycle, $k\geq 4$, embedded as a full subgraph of $\Gamma$. Using the previous results we can easily extend this to Lie algebras and we get Corollary D from the introduction.
Recall that right angled Artin Lie algebras $L_\Gamma$ are $\N$-graded with grade components $(L_\Gamma)_1=kv_1+\ldots + kv_m$ where $\{v_1,\ldots,v_m\}=V(\Gamma)$ and $(L_\Gamma)_i=[(L_\Gamma)_1,(L_\Gamma)_{i-1}]$ for $i\geq 2$.This implies $[L_\Gamma,L_\Gamma]=\oplus_{i\geq 2}(L_\Gamma)_i$.

\begin{corollary}  \label{cor-coherent} Let $L_\Gamma$ be a right angled Artin Lie algebra. The following are equivalent:
\begin{itemize}
\item[i)] $U(L_\Gamma)$ is coherent,

\item[ii)] $L_\Gamma$ is graded-coherent,

\item[iii)] the graph $\Gamma$ is chordal.
\end{itemize}
\end{corollary}
    \begin{proof} The implication from i) to ii) is Lemma \ref{graded}. Assume that $L_\Gamma$ is graded-coherent, we claim that $\Gamma$ is chordal. Otherwise, we can find some  cycle $\Delta$ of, say, $k$ vertices, $k\geq 4$, which is a full subgraph of $\Gamma$. As any graded subalgebra of $L_\Delta$ is graded-coherent, $L_{\Delta}$ is graded-coherent. But this is a contradiction: take any linear map 
    $$\chi:L_{\Delta}/[L_{\Delta},L_{\Delta}]\to k$$
     such that $\chi(v)\neq 0$ for any $v$ vertex of $\Gamma$. Let $I_\chi\normal L_{\Delta}$ be the ideal of $L_{\Delta}$ that projects onto $\ker\chi$. As 
     $$\oplus_{i\geq 2}(L_{\Delta})_i=[L_{\Delta},L_{\Delta}]\leq I_\chi,$$ $I_\chi$ is also $\N$-graded. Moreover,
  by   \cite[Corollary D]{K-M2} $I_\chi$ is  finitely generated since $\Delta$ is connected  but not finitely presented since $\Delta$ is not 1-connected.
    
Next, assume that $\Gamma$ is chordal. We claim that $U(L_\Gamma)$ is coherent. As $\Gamma$ is chordal, either $\Gamma$ is complete or there are subgraphs $\Gamma_1,\Gamma_2\subseteq\Gamma$ such that $\Gamma=\Gamma_1\cup\Gamma_2$ and $\Gamma_1\cap\Gamma_2$ is complete (see \cite{D}). In the first case $L_\Gamma$ is finite dimensional abelian so $U(L_\Gamma)$ is an abelian polynomial ring, so is Noetherian  and coherent. In the second case  we may assume by induction that $U(L_{\Gamma_1})$ and $U(L_{\Gamma_2})$ are both coherent and as $U(L_{\Gamma_1\cap\Gamma_2})$ is Noetherian we also get the result by Theorem \ref{coherentfund}.
    
    \end{proof}
    
    By \cite{SDS} the commutator subgroup of a right angled Artin group is free if and only if the the underlying graph is chordal. We show here a Lie algebra version of this fact.
    
    \begin{lemma}  Let $L_\Gamma$ be a right angled Artin Lie algebra, where  $\Gamma$ is a chordal graph. Then
    $[L_{\Gamma},L_{\Gamma}]$ is a free Lie algebra.
    \end{lemma}
    
    \begin{proof} If $\Gamma$ is a  full graph there is nothing to prove as $L_{\Gamma}$ is abelian.
    
    By \cite{Di} since $\Gamma$ is chordal there is a decomposition $\Gamma_1 \cup \Gamma_2 = \Gamma$, where $\Gamma_1 \cap \Gamma_2 = \Gamma_0$, where $\Gamma_0$ is a full graph. This gives a decomposition
    $$L_{\Gamma} = L_{\Gamma_1} *_{L_{\Gamma_0}} L_{\Gamma_2}$$
    Note that by the defining relations of a right angled Lie algebra we have that the inclusion $L_{\Gamma_i} \to L_{\Gamma}$ induces an inclusion $L_{\Gamma_i} / [L_{\Gamma_i}, L_{\Gamma_i}] \to L_{\Gamma}/[L_{\Gamma}, L_{\Gamma}]$. Hence
    $[L_{\Gamma}, L_{\Gamma}] \cap L_{\Gamma_i} = [L_{\Gamma_i}, L_{\Gamma_i}] \hbox{ for  } i = 0,1,2.$
     Then 
     by induction on the number of vertices we can assume the results holds for $L_{\Gamma_1}$  and $L_{\Gamma_2}$. Thus
     $$[L_{\Gamma}, L_{\Gamma}] \cap L_{\Gamma_i} = [L_{\Gamma_i}, L_{\Gamma_i}] \hbox{ is a free Lie algebra  for  } i = 1,2.$$
     Then by Proposition \ref{free-bass-serre} $[L_{\Gamma}, L_{\Gamma}] $ is free.
    \end{proof}
    
    \begin{proposition} \label{neue}
    Let $L$ be an $\mathbb{N}$-graded Lie algebra such that $[L,L]$ is a free Lie algebra. Then every finitely generated graded Lie subalgebra $S$ of $L$ is of homological type $FP_{\infty}$.
    \end{proposition}
    
    \begin{proof}
    Consider the short exact sequence of graded Lie algebras $0 \to F \to  S \to Q \to 0$, where  $F =S \cap [L, L]$, then $Q$ is finitely generated abelian. Consider the  Lie algebra version of the LHS spectral sequence 
    $$E^2_{i,j} = H_i (Q, H_j(F, k))$$
    that converges to $H_{i+j}(S,k)$.
    Since $F$ is free $E^2_{i,j}= 0$ for $j \geq 2$, hence the short exact sequence is concentrated in two lines. Hence
    $E^{\infty}_{i,j} = E^3_{i,j}$ is a subquotient of $E^2_{i,j}$, so $ \dim E^3_{i,j} \leq \dim  E^2_{i,j}$ and there is a short exact sequence
    $$0 \to E^3_{n-1,1} \to H_n (S, k) \to E^3_{n, 0} \to 0.$$
    Since $S$ is finitely generated we have that $H_1(S,k)$ is finite dimensional, in particular 
     $E^2_{0,1} = H_0(Q, H_1(F, k))$ is finite dimensional. Since $V = H_1(F, k) \simeq F/ [F,F]$ is a graded $U(Q)$-module via the adjoint action we know that $V$ is finitely generated as $U(Q)$-module if and only if $H_0(Q, V)$ is finite dimensional.  Note that $U(Q)$ is a Noetherian ring, so once we have that $V$ is a finitely generated as $U(Q)$-module, it is of type $FP_{\infty}$ over $U(Q)$ and hence $H_i(Q, V)$ is finite dimensional for every $i \geq 1$. In particular $E^2_{i,1} = H_i(Q, V)$ is finite dimensional for every $i \geq 0$. 
     
     Finally $E^2_{i,0} = H_i(Q, k)$ is finite dimensional since $Q$ is finite dimensional abelian Lie algebra. Thus $\dim_k(H_i(S, k)) < \infty$ for every $i$ and by Proposition \ref{homgrad2} this implies that $S$ is of homological type $FP_{\infty}$.
    \end{proof}
    
   The following result follows directly from Proposition \ref{neue}. Note it strengthens Corollary \ref{cor-coherent}.
     
    \begin{corollary}  Let $L_\Gamma$ be a right angled Artin Lie algebra, where  $\Gamma$ is a chordal graph. Let $S \leq L$ be a finitely generated $\mathbb{N}$-graded Lie subalgebra. Then $S$ is of type $FP_{\infty}$.
    \end{corollary} 
    
    It is known that an ascending HNN-extension of a free group is coherent \cite{F-H}. As a corollary of Proposition \ref{neue}
    we show that the same holds for Lie algebras. We call an HNN-extension Lie algebra  $W =\langle L ,t | [t,a]=d(a)\text{ for }a\in  A\rangle$ ascending if $A = L$.
    
    \begin{lemma} An ascending HNN-extension of a free Lie algebra is graded-coherent.
    \end{lemma}
    \begin{proof}
    By assumption $W =\langle L ,t | [t,a]=d(a)\text{ for }a\in  A\rangle$  with $L = A$ free Lie algebra. Thus
    $A$ is an ideal of codimension 1 in $W$, hence $[W, W] \subseteq A$ is free. Then we can apply Proposition \ref{neue}.
    \end{proof}

    By an analogy with the case of groups we may ask:
    
    \begin{question} Assume that $L$ is the fundamental Lie algebra of a graph of Lie algebras such that the vertex Lie algebras $L_v$ are coherent  and the edge algebras $L_e$ have the property that every subalgebra is finitely generated. Is then $L$ coherent? (again, this is true for groups, see \cite{Wilton}).\end{question} 
    
    \section{An example} \label{example}

    Let
    $$L = M * N, \hbox{ where } M = ka \oplus k b, N = k x \hbox{ are abelian Lie algebras }$$
    Thus there is a presentation ( in terms of generators and relations) $$L = \langle a,b,x ~| ~ [a,b]  \rangle.$$ Note that $L$ is a $\mathbb{N}$-graded Lie algebra, where $a,b,x$ all have degree 1.
    Let $S$ be the $\mathbb{N}$-graded Lie subalgebra $$S = \langle a,b, z = [x,a], t = [x,b] \rangle \leq L$$
    \begin{lemma} \label{pres}  
    
    a) $a,b,z,t$ is a minimal set of generators of $S$ and $\dim_k  H_2(S,k) = 2$.
    
    b) $S$ has a presentation $\langle a,b,z,t  | [z,b] - [t,a], [a,b] \rangle$.
    \end{lemma}
    
    \begin{proof}
   a) Consider first the ideal $J$ of $L$ generated by $x$. Note $J$ is a graded ideal of $L$ that as a Lie algebra is generated by the homogeneous elements $$T = \{ ad(a)^i ad(b)^j (x) ~| ~i,j \geq 0\}$$
    where $ad(y)(s) = [s,y]$. Furthermore by Proposition \ref{free-bass-serre} $J$ is a free Lie algebra and if $T$ are linearly independent in $J/ [J,J]$ we can conclude that $T$ is a free basis of $J$. To show this consider the Lie algebra $M_0 = U(M) \leftthreetimes M$, where the adjoint action of $M$ on  the ideal $U(M)$ of $M_0$ is given by right multiplication.  Note that $M_0$  is a quotient of $L$ via the homomorphism $\varphi$ that is identity on $a$ and $b$ and sends $x$ to $ 1 \in U(M)$. Note that $\varphi(J) = U(M)$ is abelian hence $J/ [J,J]$ maps surjectively to $U(M)$ and $\varphi(T)$ is a basis of $U(M)$ as a $k$-vector space. Hence $T$ is linearly independent in $J/ [J,J]$ as required.
    
    Consider $J_0 = J \cap S$. Thus we have a split short exact sequence of Lie algebras \begin{equation} \label{split1} 0 \to J_0 \to S \to M \to 0 \end{equation}  Note that $J_0$ as a Lie algebra is generated by $T_0 = T \setminus \{ x \}$, hence $J_0$ is a free Lie algebra with a free basis $T_0$.
    Then
    $J_0/ [J_0, J_0]$ is a $k$-vector space with  basis $T_0$. Note that $ad(a)^i ad(b)^{j-1} (t) = ad(a)^i ad(b)^j (x) = ad(a)^{i-1} ad(b)^j (z)$, hence
      $J_0/ [J_0, J_0]$  is a right $U(M)$-module generated by $z$ and $t$ subject to the relation $ad(a) (t) = ad(b) (z)$. Furthermore
      $$J_0/ [J_0, S] \simeq (J_0/ [J_0, J_0]) \otimes_{U(M)} k \simeq k \oplus k. $$ Since $S$ is a split extension of $J_0$ by $M$ and $M$ is abelian we conclude that
      $$S/[S, S] \simeq J_0/ [J_0, S] \oplus M \hbox{ has dimension 4}.$$
      Since $a,b,z,t$ is a generating set of $S$ we deduce that $a,b,z,t$ is a minimal generating set of $S$.
      
      Consider the spectral sequence
      $$E^2_{i,j}  = H_i(M, H_j(J_0, k))$$
      converging to $H_{i+j} (S,k)$.
      Since $J_0$ is free $H_j(J_0, k) = 0$ for $j \geq 2$. Hence $E^{\infty}_{i,j} = 0$ for $j \geq 2$.
      Since (\ref{split1}) splits the map
      $$\mu : H_2(S, k) \to H_2(M,k) = E^2_{2,0}$$ is a split epimorphism. By the convergence of the spectral sequence there
       is an exact sequence
      $$
     0\to E^\infty_{1,1} \to H_2(S,k) \to E^{\infty}_{2,0}     \to 0$$
      and $E^{\infty}_{2,0} = E^3_{2,0}$ is precisely the image of $\mu$. Thus $E^{\infty}_{2,0} = E^{3}_{2,0} = E^{2}_{2,0}$. 
      
      By the bi-degrees of the differentials and since $E_{3,0}^{ \infty} = 0$ we have $$E^\infty_{1,1}  = E^2_{1,1} = H_1(M, H_1(J_0, k))\simeq H_1(M, J_0/ [J_0, J_0]).$$ Let $V = J_0/ [J_0, J_0]$. Recall that $V$ is is a right $U(M)$-module generated by $z$ and $t$ subject to the relation $ad(a) (t) = ad(b) (z)$. Hence there is a short exact sequence of right $U(M)$-modules
      $$0 \to U(M) \to U(M) \oplus U(M) \to V \to 0$$Then the long exact sequence in homology implies
      $$ \ldots \to 0 = H_1(M, U(M) \oplus U(M)) \to H_1(M, V) \to U(M) \otimes_{U(M)} k \to$$ $$ (U(M) \oplus U(M)) \otimes_{U(M)} k \to V \otimes_{U(M)} k \to 0$$that can be  rewritten as 
$$ \ldots \to 0 = H_1(M, U(M) \oplus U(M)) \to H_1(M, V) \to k \to k^2 \to k^2 \to 0$$
Hence $H_1(M, V) \simeq k$ and
$$\dim_k H_2(S,k) = 
\dim E_{1,1}^2 + \dim E_{2,0}^2 = 
\dim_k H_1(M, V) + \dim_k H_2(M, k) = 1 + 1 = 2$$

b) Since $a,b,z,t$ is a minimal generating set by Lemma \ref{presentation} we conclude that $S$ has a presentation
$\langle a,b,z,t ~| ~R_0 \rangle$ where $|R_0| = 2$ is minimal possible. Note that  the elements of  $R_0$ should be of minimal possible degrees (as we work with  the $\mathbb{N}$-graded Lie algebra $S$). On other hand it is easy to verify by hand that
$[z,b] - [t,a], [a,b] $ are the relations of the smallest degree ( recall that $a,b$ have degree 1 and $z,t$ have degree 2).
Hence we get that $R_0 $ can be chosen $\{ [z,b] - [t,a], [a,b] \}$.
    \end{proof}
    
    \begin{lemma} \label{free-product1}  Let $M,N,L,S$ be as before. Then $S$ is a Lie subalgebra of $L = M * N$ that contains $M$ but $S \not= M * Q$ for  any Lie algebra $Q$.
    \end{lemma}
    
    \begin{proof}
    Suppose $S = M * Q$ and let $I$ be the ideal of $S$ generated by $M$. Then using Lemma \ref{pres}  we have
    $$Q \simeq S / I = \langle a,b,z,t  | [z,b] - [t,a], [a,b], a,b \rangle = \langle z, t | \emptyset \rangle$$
    is a free Lie algebra. 
    Then
    $$2 = \dim_k H_2(S, k) = \dim_k H_2(M * Q, k) =$$ $$ \dim_k(H_2(M,k) \oplus H_2(Q, k) )= \dim_k H_2(M, k) = 1$$
    a contradiction.
    \end{proof}
    
     \begin{lemma} \label{lemma-long}
    Consider the  Lie algebras $S = \langle x_1, x_2, x_3, x_4 ~| [x_1, x_2], [x_3, x_2] + [x_1, x_4]\rangle$, $\widetilde{S} = \langle x_1, x_2, x_3, x_4 ~| [x_1, x_2], [x_1, x_3] \rangle$
    and
    $\widehat{S} = \langle x_1, x_2, x_3, x_4 ~| [x_1, x_2], [x_4, x_3] \rangle$. Then the nilpotent Lie algebras $E = S/ [S,S,S], \widetilde{E} = \widetilde{S}/ [\widetilde{S}, \widetilde{S}, \widetilde{S}]$ and
     $\widehat{E} = \widehat{S} /[\widehat{S}, \widehat{S} ,\widehat{S} ]$  are pairwise non-isomorphic. In particular the Lie algebras $S$, $\widetilde{S}$ and $\widehat{S}$ are pairwise non-isomorphic.
    \end{lemma}
    
    \begin{proof} We write $E = V \oplus [E, E]$, $\widetilde{E} = V \oplus [ \widetilde{E}, \widetilde{E}]$ and $\widehat{E} = V \oplus [\widehat{E}, \widehat{E}]$, where $V $ is the vector space spanned by $x_1, x_2, x_3, x_4$.
    
    1) It is easy to see that if there is an isomorphism $\varphi$ between two of the Lie algebras $E$, $\widetilde{E}$ and $\widehat{E}$ then there is a $\mathbb{N}$-graded isomorphism  $\varphi_0$ between the same Lie algebras i.e. is induced by an automorphism of $V$  as a vector space. For  example if $\varphi(x_i) = v_i + w_i$, where $v_i \in V$ and $w_i$ is an element of the corresponding derived Lie subalgebra then we can define $\varphi_0(x_i) = v_i$.
    
    2) Let $I$ be the ideal of $\widetilde{E}$ generated by $x_1$. Then $\widetilde{E}/ I = : F$ is a 3-generated free nilpotent $\mathbb{N}$-graded Lie algebra of class 2.
    
    2.1) Suppose that $\widetilde{E}$ and $E$ are isomorphic as $\mathbb{N}$-graded Lie algebras. Then there is an ideal $J$ of $E$ generated by one element of $V$ such that $E/ J \simeq F$. Let $\pi : E \to E/J$ be the canonical epimorphism. Since $[x_1, x_2] = 0$ 
    we have that $[\pi(x_1), \pi(x_2)] = 0$. Since $F$ is free nilpotent with $\pi(x_1), \pi(x_2)$ either zero or of degree 1 we deduce that either   $\pi(x_1) = \lambda \pi(x_2)$ for some $\lambda \in k$ or  $\pi(x_2) = 0$. Recall that
    $[x_3, x_2] + [x_1, x_4] = 0$ in $E$, hence
    \begin{equation} \label{abc} 0 = [\pi(x_3), \pi(x_2)] + [\pi(x_1), \pi(x_4)]. \end{equation}

 2.1.1) If $\pi(x_1 - \lambda x_2) = 0$ then $J$ is generated as an ideal by $x_1 - \lambda x_2$ and $F$ has free generators $\pi(x_2), \pi(x_3), \pi(x_4)$. By (\ref{abc}) $0 = [\pi(x_3), \pi(x_2)] + [\pi(x_1), \pi(x_4)] =[\pi(x_3), \pi(x_2)] + [\lambda \pi(x_2), \pi(x_4)]  = [\pi(x_2), \lambda \pi(x_4) - \pi(x_3)]$. Then we get a contradiction with $F$ a free nilpotent Lie algebra of class 2 with free generators $\pi(x_2),$ $ \pi(x_3), \pi(x_4)$.

  2.1.2) If $\pi(x_2) = 0$ then $J$ is generated as an ideal by $x_2$. And we can argue as in case 2.1.1.

  2.2) Suppose that $\widetilde{E}$ and $\widehat{E}$ are isomorphic as $\mathbb{N}$-graded Lie algebras. Then there is an ideal $\widehat{J}$ of $\widehat{E}$ generated by one element of $V$ such that $\widehat{E}/ \widehat{J} \simeq F$. Let $\pi : \widehat{E} \to \widehat{E}/\widehat{J}$ be the canonical epimorphism. Since $[x_1, x_2] = 0$ in $\widehat{E}$
    we have that $[\pi(x_1), \pi(x_2)] = 0$ and as in case 2.1 we deduce that either  $\pi(x_1) = \lambda \pi(x_2)$ for some $\lambda \in k$ or $\pi(x_2) = 0$.
    Thus for some $a \in \{ x_2 , x_1 - \lambda x_2 \} \subset V$ we have that $\pi(a) = 0$.
    Similarly since 
    $[x_3, x_4] = 0$ in $E$, 
    there is some   $b \in \{ x_4 , x_3 - \mu x_4 \} \subset V$ for some $\mu \in k$ such that $\pi(b) = 0$. Thus $Ker (\pi) \cap V$  is a vector space over $k$ of dimension at least 2.  This contradicts the fact that $Ker ( \pi) = \widehat{J}$ is an ideal of $\widehat{E}$  generated by one element of $V$.

  2.3) Assume now that there is an $\mathbb{N}$-graded isomorphism between $E$ and $\widehat{E}$. Then in $V \subseteq 
  E$ there are elements $y_1, y_2, y_3, y_4$ that are linearly independent and such that $[y_1, y_2] = 0 = [y_3, y_4]$. 
  Then there are elements $v_1, v_2$ in $E$ such that $k v_1 \not= k v_2$ ,  either $v_1$ or $v_2$ does not belong to $k x_1 + k x_2$ and $[v_1, v_2]=0$. 
  
  Write
  $v_1 = \sum_{1 \leq i \leq 4} a_i x_i, v_2 = \sum_{1 \leq i \leq 4} b_i x_i$, where all $a_i, b_i \in k$.
  Then $ 0 = [v_1, v_2] = \sum_{ 1 \leq i < j \leq 4}(a_i b_j - a_j b_i)[x_i, x_j]$.
  Using that $[x_1, x_2] =0, [x_3, x_2] + [x_1, x_4]=0$ in $E$ we deduce that 
  $$0 =\sum_{ 1 \leq i < j \leq 4}(a_i b_j - a_j b_i)[x_i, x_j] = 
  (a_1 b_3 - a_3 b_1) [x_1, x_3] + $$ $$(a_1 b_4 - a_4 b_1 + a_2 b_3 - a_3 b_2) [x_1, x_4] + (a_2 b_4 - a_4 b_2) [x_2, x_4] + ( a_3 b_4 - a_4 b_3) [x_3, x_4].$$
  Since $[x_1, x_3], [x_1, x_4], [x_2, x_4], [x_3, x_4]$ are linearly independent in $E$ we deduce that
  $$a_1 b_3 - a_3 b_1= 0, a_2 b_4 - a_4 b_2 = 0, a_3 b_4 - a_4 b_3 = 0, a_1 b_4 - a_4 b_1 + a_2 b_3 - a_3 b_2 =0.$$ 
  Recall that  either $v_1$ or $v_2$ does not belong to $k x_1 + k x_2$, say $v_2$. Then either $b_3 \not=0$ or $b_4 \not= 0$. Without loss of generality $b_3 \not= 0$. Then
  $$a_1 = \frac{a_3 b_1}{b_3}, a_4 = \frac{a_3 b_4}{b_3}, a_2 b_4 = a_4 b_2 = \frac{a_3 b_4 b_2}{b_3}.$$
  If $b_4 \not= 0$ then $a_2 = \frac{a_3 b_4 b_2}{b_3 b_4} = \frac{a_3 b_2}{b_3}$, hence $a_i = \frac{a_3}{b_3} b_i$ and $v_1 = \frac{a_3}{b_3} v_2$, a contradiction.
  
  If $b_4 = 0$ then $a_4 = 0$, $0 = a_1 b_4 - a_4 b_1 + a_2 b_3 - a_3 b_2 = a_2 b_3 - a_3 b_2 $ hence $a_2 = \frac{a_3 b_2}{b_3}$. Hence $a_i = \frac{a_3}{b_3} b_i$ and $v_1 = \frac{a_3}{b_3} v_2$, a contradiction. 
    \end{proof} 
    
    The following result shows that even in the cases of $\mathbb{N}$-graded Lie algebras Kurosh type result does not hold.
    
    \begin{lemma}   Let $M,S$ be as before. Then $S \not= C * D$ for  any non-zero  $\mathbb{N}$-graded Lie subalgebras $C$ and $D$ of $S$.
    \end{lemma}
    \begin{proof}  Suppose $S = C * D$.
Recall that  $M\leq S$ is abelian. If $M$ is not a subalgebra of $C$ or of $D$ we have that $\dim_k(C \cap M) \leq 1$, $\dim_k(D \cap M) \leq 1$ hence by Proposition \ref{free0} $M$ is a free Lie algebra, a contradiction.

Suppose from now that $M$ is a Lie subalgebra of $C$. Recall that  by Lemma  \ref{pres} $S = \langle a,b,z,t  | [z,b] - [t,a], [a,b] \rangle$.
If $C$ is generated by 2 elements since $C$ is $\mathbb{N}$-graded i.e. $C = \oplus_{i \geq 1}  C_i$, $M = k a \oplus k b \subseteq C_1$  we deduce that $C = M$, a contradiction with Lemma \ref{free-product1}.

If $ C$ is generated by at least 3 elements, since $ S=C*D$ is generated by precisely 4 elements we conclude that $C$ is generated by 3 elements and $D$ by 1 element. Then $D = k w$ and since  $S$ has 2 relators, $C = \langle a,b, v ~| ~[a,b], r \rangle$ for some relation $r$ on $a,b,v$. 
 Since $S$ has a presentation with quadratic relations the same holds for $C$, so $r = \alpha[a,v] + \beta[b,v] = [\alpha a + \beta b, v]$ for some $\alpha, \beta \in k$ not both zero. Suppose $\alpha \not= 0$. Then for $\widehat{a} = \alpha a + \beta b$ we have $C = \langle \widehat{a}, b, v ~| ~[ \widehat{a}, b], [\widehat{a}, v] \rangle$ and $C * D$ is isomorphic to the Lie algebra $\widetilde{S}$ from Lemma \ref{lemma-long}. By Lemma \ref{lemma-long} $S$ and $\widetilde{S}$ are not isomorphic, a contradiction.

    \end{proof}
     In \cite{Feldman2} it was shown an example of a Lie algebra with infinitely many ends that is not a free product. The example from \cite{Feldman2} is derived from the Shirshov example of   a Lie algebra that is not a free product \cite{Shirshov2}, thus it is not $\mathbb{N}$-graded. Next we show that there is a $\mathbb{N}$-graded example.
    
    \begin{corollary}
    There exists a $\mathbb{N}$-graded Lie algebra $L_0$ that has infinitely many ends i.e. $\dim_k H^1(L_0, U(L_0)) = \infty$ but $L_0$ is not a free product of $\mathbb{N}$-graded Lie subalgebras.
    \end{corollary}
    
    \begin{proof} By \cite[Thm.~1.4]{Feldman2} for any Lie subalgebra $L_0$ of $L_1 * L_2$  such that $L_0 $ is not a subalgebra of $L_i$ for $i = 1, 2$ we have that $\dim_k H^1(L_0, U(L_0)) = \infty$. We apply this for $L_0 = S$, $L_1 = M, L_2 = N$.
    \end{proof}
    \begin{lemma} \label{final}
    $S$ is not a right angled Artin Lie algebra.
    \end{lemma}
    
    \begin{proof} Suppose that $S$ is a right angled Artin Lie algebra. Since $\dim_k H_1(S,k) = 4$ and $\dim_k H_2(S, k) = 2$ the underlying graph of the 
    right angled Artin Lie algebra $S$ has exactly  4 vertices and 2  edges. Thus  up to isomorphism there are two options for $S$ :
    $$\widetilde{S} = \langle x_1, x_2, x_3, x_4 ~| [x_1, x_2], [x_1, x_3] \rangle
  \hbox{   or }
    \widehat{S} = \langle x_1, x_2, x_3, x_4 ~| [x_1, x_2], [x_4, x_3] \rangle.$$
    By Lemma \ref{lemma-long} any two of the Lie algebras $S/[S,S,S]$, $\widetilde{S}/[\widetilde{S},\widetilde{S},\widetilde{S}]$ and $\widehat{S}/[\widehat{S},\widehat{S},\widehat{S}]$ are not isomorphic.
    \end{proof}
    
    Observe that $L = M* N$ is a right angled Artin Lie algebra whose underlying graph $\Gamma$ is one edge and a separate vertex. Note that the square or the line with 4 vertices do not embed in $\Gamma$. By \cite{D2}  the right-angled Artin group $G_{\Gamma}$ has the property that every finitely generated subgroup is a RAAG. The Lie algebra version of this results does not hold by Lemma \ref{final}. The reason this happens is that in the case of Lie algebras there is no Kurosh type theorem.

 \end{document}